\documentclass{article}

\usepackage{amsmath}
\usepackage{amssymb} 
\usepackage{amsfonts} 
\usepackage{bm}
\usepackage{graphicx}
\usepackage{booktabs}
\usepackage[margin=1in]{geometry}
\usepackage{natbib}
\usepackage{amsthm}
\usepackage[colorlinks]{hyperref}
\hypersetup{citecolor=blue}
\usepackage{authblk} 
\usepackage[capitalize]{cleveref}

\usepackage{xcolor}

\newtheorem{definition}{Definition}
\newtheorem{theorem}{Theorem}
\newtheorem{lemma}{Lemma}
\newtheorem{proposition}{Proposition}
\newtheorem{corollary}{Corollary}

\newtheorem*{theorem*}{Theorem}
\newtheorem*{lemma*}{Lemma}
\newtheorem*{proposition*}{Proposition}
\newtheorem*{corollary*}{Corollary}
\newtheorem*{example*}{Example}
\newtheorem{assumption}{Assumption}

\newcommand{\prox}{\operatorname{prox}}
\newcommand{\dist}{\operatorname{dist}}
\newcommand{\argmin}{\operatorname{argmin}}

\newcommand{\dom}{\operatorname{dom}}

\title{Trajectory-Restricted Optimization Conditions\\ and Geometry-Aware Linear Convergence}

\author[1]{Faris Chaudhry}
\author[1]{Anthea Monod}
\author[2]{Keisuke Yano}
\affil[1]{Imperial College London, UK\\ \texttt{\{faris.chaudhry22, a.monod\}@imperial.ac.uk}}
\affil[2]{Institute of Statistical Mathematics, Japan\\ \texttt{yano@ism.ac.jp}}

\date{}

\begin{document}

\maketitle

\begin{abstract}
Linear convergence of first-order methods is typically characterized by global optimization conditions whose constants reflect worst-case geometry of the ambient space. In high-dimensional or structured problems, these global constants can be arbitrarily conservative and fail to capture the geometry actually encountered by optimization trajectories. In this paper, we develop a trajectory-restricted framework for linear convergence based on localized geometric regularity. We introduce restricted variants of the Polyak--{\L}ojasiewicz inequality, error bound, and quadratic growth conditions that are required to hold only on subsets of the domain. We show that classical convergence guarantees extend under these localized conditions, and in key cases, we develop new arguments that yield explicit relationships between the corresponding constants. The resulting rates are governed by geometric quantities associated with the regions traversed by the algorithm. For polyhedral composite problems, we prove that convergence is controlled by restricted Hoffman constants corresponding to the active polyhedral faces visited along the trajectory. Once the iterates enter a well-conditioned face, the effective condition number improves accordingly. Our work provides a geometric quantification for fast local convergence after active-set or manifold identification and more broadly suggests that linear convergence is fundamentally governed by the geometry of the subsets explored by the algorithm, rather than by worst-case global conditioning.
\end{abstract}

\section{Introduction}

Linear convergence of first-order methods is commonly established through global regularity conditions such as the Polyak--{\L}ojasiewicz (PL) inequality or proximal error bounds. These conditions are now known to be essentially minimal for guaranteeing linear rates across smooth, nonsmooth, and even certain nonconvex composite problems~\citep{Karimi2016PL,Bolte2017ErrorBounds,Drusvyatskiy2018ErrorBounds}. In this classical framework, convergence rates are governed by global geometric constants that quantify curvature or metric regularity of the objective over the entire ambient space. However, much of the existing literature focuses primarily on the existence of such bounds rather than the exactness of the resulting convergence rates. Even when considering explicit convergence rates, these are often global rates that may not be tight. For example, in high-dimensional settings, global constants often scale with the ambient dimension of the problem, leading to theoretical rates that are significantly looser than the fast convergence observed empirically. This discrepancy suggests that global analysis fails to capture the favorable local geometry encountered by optimization trajectories.

In this paper, we develop a trajectory-restricted framework for linear convergence that localizes classical regularity conditions to subsets of the domain. We consider composite problems of the form 
\begin{equation*}
    F(x) = f(x) + g(x),
\end{equation*}
where $f$ is an $L$-smooth (and potentially nonconvex) function, and $g$ is convex and potentially nonsmooth. We introduce restricted variants of the proximal PL inequality and related error bound conditions that are required to hold only on subsets of the domain traversed by the optimization trajectory. Standard convergence proofs extend under these localized conditions, yielding linear rates governed by geometric constants associated with the regions actually explored by the algorithm.

For polyhedral composite problems, our localization reveals a particularly transparent geometric mechanism. Once the iterates enter an \emph{active} polyhedral face, such as the optimal support of a sparse estimator, the relevant conditioning is determined by a restricted \emph{Hoffman constant} associated with that face, rather than by the global constant of the full system. In high-dimensional regimes, this leads to a substantial improvement in the effective condition number and explains the fast local convergence observed after active-set or manifold identification.  More broadly, our work suggests that linear convergence is fundamentally governed not by global worst-case geometry, but by the geometric structure of the subsets encountered along the optimization trajectory.  
Our framework allows for the explicit quantification on how convergence rates improve after manifold identification by relating the local convergence rate to the Hoffman constant of the active set. More broadly, it provides a unified lens through which optimization convergence, polyhedral geometry, and high-dimensional statistical conditions can be understood.\\

To develop this framework, \cref{sec:background} establishes the necessary background, reviews related literature, and formally specifies our problem of interest. \cref{sec:theory} transitions from the standard PL inequality and error bounds to our proposed restricted variants, detailing their relationship and demonstrating how standard linear convergence guarantees adapt to this localized framework. \cref{sec:lasso-svm-applications} then grounds this theory in polyhedral systems, deriving restricted Hoffman constants for the least absolute shrinkage and selection operator (LASSO)~\citep{Tibshirani1996Lasso} and support vector machines (SVMs)~\citep{Boser1992MarginClassifiers, Cortes1995SupportVector}. Finally, \cref{sec:discussion} concludes by exploring the wider implications of geometry-aware convergence and trajectory containment.

\section{Background}
\label{sec:background}

The linear convergence of first-order methods is typically established through geometric regularity conditions that relate function values, gradients, and distances to the solution set. In this section, we provide an overview of existing related work as well as the formal setup for our study.

\subsection{Related Work}

Geometric regularity conditions have been studied in various contexts with different motivations and purposes, depending on the field of study.  Here, we review their occurrence in classical settings, high-dimensional statistics, and optimization.

\paragraph{Classical Regularity Conditions and Linear Convergence.} Classical examples of geometric regularity conditions for the linear convergence of first-order methods include the Polyak--{\L}ojasiewicz inequality, proximal error bounds (EB), and metric regularity \citep{Polyak1963Inequality, Lojasiewicz1963topological}. These conditions provide quantitative relationships of the form
\begin{equation*}
    F(x) - F^* \lesssim \|\nabla F(x)\|^2_{2} \quad \text{or} \quad  \mathrm{dist}(x,\mathcal{X}^*):=\inf_{y\in\mathcal{X}^*}\|x-y\|_{2} \lesssim \|\nabla F(x)\|_{2},
\end{equation*}
and are known to be essentially minimal assumptions for ensuring linear convergence of gradient and proximal methods, even in certain nonconvex and nonsmooth settings~\citep{Bolte2007KLSubanalytic,Attouch2013DescentMethodsSemiAlgebraic,Luo1993ErrorBounds}. Such conditions originated from gradient dominance inequalities for smooth problems and have been extended to broad classes of nonsmooth and structured objectives via the Kurdyka--{\L}ojasiewicz framework, which provides a unifying lens for analyzing descent methods~\citep{Lojasiewicz1963topological, Kurdyka1998, Attouch2013DescentMethodsSemiAlgebraic}. Closely related formulations appear in the form of error bounds for feasible and composite optimization problems, where distances to solution sets are controlled by residuals~\citep{Drusvyatskiy2018ErrorBounds, Zhou2017ErrorBounds, Luo1993ErrorBounds, Pang1997ErrorBounds}. In nonsmooth optimization, recent work has clarified how such conditions interact in increasingly general settings: \citet{drusvyatskiy2021nonsmooth} study Taylor-like model-based methods and relate step sizes, slope error bounds, quadratic growth, and termination criteria, while \citet{Liao2024PMLR} establish implication and equivalence relationships among error bounds, PL, and quadratic growth for weakly convex functions and use them to derive linear convergence of proximal point methods.

In polyhedral optimization, these relationships admit a concrete geometric interpretation through \emph{Hoffman-type} error bounds, where the distance to a feasible set defined by linear constraints is controlled by the constraint violation \citep{Hoffman1952ApproximateSolutions}. The associated constants can be understood as measures of conditioning of the underlying constraint system, which links convergence rates directly to the geometry of the feasible region. More generally, local or level-set error bounds have been studied for a wide range of structured problems, including non-polyhedral settings such as low-rank and nuclear-norm regularization \citep{Zhou2017ErrorBounds, Drusvyatskiy2018ErrorBounds}. However, these analyses typically establish the existence of regularity constants independently of the optimization trajectory and often rely on global worst-case quantities. As a consequence, the resulting convergence rates may be overly conservative in high-dimensional settings. A key insight is that such geometric conditions are sufficient and, in fact, minimal for linear convergence \citep[e.g.,][]{Karimi2016PL,Bolte2017ErrorBounds,Drusvyatskiy2018ErrorBounds}, but the constants governing these rates remain tied to global properties of the problem.

\paragraph{High-Dimensional Statistics and Restricted Regularity.} A common feature of these classical results is that the associated constants are defined globally and thus reflect the worst-case geometric properties of the ambient space. As a result, the convergence rates may be overly pessimistic, particularly in high-dimensional regimes where the ambient dimension $d$ greatly exceeds the sample size 
$n$ and global conditioning deteriorates with dimension. In such settings, global strong convexity typically fails and meaningful guarantees must instead rely on structure.

This observation has motivated the study of localized or restricted notions of regularity. In high-dimensional statistics, for example, conditions such as the restricted eigenvalue (RE) property, restricted strong convexity (RSC), and the restricted isometry property (RIP) quantify curvature only along structured subsets of the parameter space, such as sparse cones. These conditions have been central to establishing statistical consistency and recovery guarantees for estimators such as the LASSO \citep{Candes2005DecodingLP,Bickel2009LassoDantzigSelector,Negahban2009UnifiedMEstimators,Wainwright2019HighDimensional,Vershynin2018HighDimensional}. Moreover, beyond their statistical role, restricted curvature has also been shown to govern the algorithmic convergence of first-order methods in these settings \citep{Agarwal2012FastGlobal}.

While such restricted notions are now common in the statistical literature, their role in optimization has been comparatively less developed. In particular, existing analyses typically relate convergence rates to properties of the objective or data matrix \citep[e.g.,][]{Wainwright2019HighDimensional}, rather than to the geometry of the regions explored by the optimization trajectory, which motivates a trajectory-aware perspective on regularity.  In such a setting, geometric constants are defined not globally but on structured subsets relevant to the iterates of the algorithm.

\paragraph{Optimization: Manifold Identification and High-Dimensional Geometry.} From an optimization perspective, a related but distinct phenomenon is that of \emph{manifold identification}. For a broad class of nonsmooth problems, including those with polyhedral regularizers such as the $\ell_1$-norm, first-order methods identify an \emph{active manifold} or support in finite time, after which the iterates evolve within a lower-dimensional structure \citep{HareLewis2004Active,Lee2012ManifoldIdentification}. This transition induces a qualitative change in the geometry of the problem: once the active set is identified, the optimization effectively proceeds on a restricted subspace and local convergence behavior reflects the conditioning of this reduced system. Under suitable nondegeneracy conditions, such as strict complementarity, this identification occurs after finitely many iterations~\citep{Sun2019ManifoldIdentification}. For instance, \citet{Nutini2019ActiveSet} report that the number of iterations required to identify the support scales as $\mathcal{O}(\log((1/t + L)/\delta))$, where $t$ is the step size, $L$ is the Lipschitz constant, and $\delta$ is the strict complementarity margin. Conversely, if the problem is degenerate, guaranteeing manifold identification in finite time is generally impossible \citep{Burke1988ActiveConstraints, Lee2012ManifoldIdentification}.
Recently, \cite{davis2025active} show that first-order methods evolves asymptotically along with the identified active structure, 
on which the dynamics can be described in the Riemannian geometry.

While these results explain why optimization may enter a fast local regime, existing analyses typically establish the existence of improved rates without explicitly linking them to geometric conditioning constants associated with the identified structure. In particular, the connection between manifold identification and the constants governing linear convergence remains largely implicit.

Our work builds on these perspectives by introducing a trajectory-restricted viewpoint on geometric regularity. Rather than relying on global constants, we consider error bounds and conditioning restricted to subsets of the domain that are relevant to the optimization trajectory, such as active sets, polyhedral faces, or structured cones. In the polyhedral setting, this leads naturally to \emph{restricted Hoffman constants}, which retain the geometric interpretability of classical Hoffman bounds rooted in the conditioning of linear systems via singular values \citep{Hoffman1952ApproximateSolutions,Robinson1981PolyhedralMultifunctions} while adapting to the local structure explored by the algorithm. In contrast to recent notions such as relative Hoffman bounds \citep{Pena2021Hoffman}, which admit only polyhedra as reference sets, our formulation preserves the ambient geometry while localizing the analysis.

\subsection{Problem Setup and the Proximal Gradient Method}

We consider composite optimization problems of the form
\begin{equation}
    \label{eq:composite-problem}
    \min_{x \in \mathbb{R}^d} F(x) := f(x) + g(x),
\end{equation}
where $f: \mathbb{R}^d \to \mathbb{R}$ is a smooth function and $g: \mathbb{R}^d \to \mathbb{R} \cup \{+\infty\}$ is a convex, possibly nonsmooth regularizer. This formulation captures a wide range of problems in optimization, statistics, and machine learning, including constrained optimization (via the indicator function $g(\cdot) = \delta_{\mathcal{X}}(\cdot)$ over a convex set $\mathcal{X}$) and regularized estimators such as the LASSO.

Throughout the paper, we work under the following standard assumptions.

\begin{assumption}[Regularity of the Objective]
    \label{ass:regularity}
    \leavevmode
    \begin{enumerate}
        \item[(i)] The function $f$ is continuously differentiable and $L$-smooth, i.e., $\|\nabla f(x) - \nabla f(y)\|_2 \le L \|x - y\|_2$ for all $x, y \in \mathbb{R}^d$.
        \item[(ii)] The function $g$ is a proper, closed, and convex.
        \item[(iii)] The optimal set $\mathcal{X}^* := \argmin_{x} F(x)$ is nonempty and the optimal value $F^* = \min_{x} F(x)$ is finite.
    \end{enumerate}
\end{assumption}

\paragraph{Proximal Gradient Method.} To solve~\eqref{eq:composite-problem}, we consider the \emph{proximal gradient method}, a standard first-order algorithm for composite problems. Given a step size $1/L$, the iterates are defined by
\begin{equation}
    \label{eq:prox-grad-update}
    x_{k+1} = \prox_{g/L} \left( x_k - \frac{1}{L} \nabla f(x_k) \right),
\end{equation}
where the \emph{proximal operator} $\prox_{\lambda g}$ is given by
\begin{equation*}
    \prox_{\lambda g}(x) := \argmin_{y \in \mathbb{R}^d} \left\{ g(y) + \frac{1}{2\lambda} \|y - x\|_2^2 \right\}.
\end{equation*}

This update can be interpreted as a forward gradient step on $f$ followed by a backward proximal step or regularization step induced by $g$.

\paragraph{Optimality Conditions (KKT).} For the convex composite problem $\min_x F(x) := f(x) + g(x)$, a point $x^*$ is \emph{optimal} if and only if it satisfies the first-order optimality (Karush--Kuhn--Tucker, KKT) condition~\citep{Boyd2004ConvexOptimization, Beck2017FirstOrderMethods}:
\begin{equation*}
    0 \in \nabla f(x^*) + \partial g(x^*),
\end{equation*}
where $\partial g(x^*)$ denotes the subdifferential of $g$ at $x^*$. 
This condition characterizes the solution set and provides a natural notion of optimality residual, which will be central to our analysis.

\paragraph{Measuring Optimality in the Nonsmooth Setting.} In smooth optimization, convergence is typically quantified in terms of the gradient norm. In the composite setting, convergence is measured by generalized notions of gradient size that account for the nonsmooth term $g$.

An important quantity in our analysis is the \emph{proximal gradient mapping}~\citep{Luo1993ErrorBounds},
\begin{equation}
    \label{eq:gradient-map}
    \mathcal{G}_\lambda(x) := \frac{1}{\lambda} \left( x - \prox_{\lambda g} \left( x - \lambda \nabla f(x) \right) \right),
\end{equation}
which measures the discrepancy between the current iterate and the result of one proximal gradient step. In particular, a point $x$ is a stationary point (a fixed point of the proximal update) if and only if $\mathcal{G}_\lambda(x) = 0$.

In nonsmooth settings, the following \emph{generalized gradient size} \citep{Karimi2016PL} admits a natural interpretation as the predicted decrease in the objective after one proximal gradient step of size $\alpha$ and plays a key role in characterizing linear convergence through proximal Polyak--{\L}ojasiewicz-type inequalities.

\begin{definition}[\citet{Karimi2016PL}]
The \emph{generalized gradient size} is given by
\begin{equation*}
    \mathcal{D}_g(x, \alpha) := - 2 
    \alpha \min_{y \in \mathbb{R}^d} \left[ \langle \nabla f(x), y - x \rangle + \frac{\alpha}{2} \|y - x\|_2^2 + g(y) - g(x) \right].
\end{equation*}
\end{definition}

An important property of the generalized gradient size is that $\mathcal{D}_g(x, \alpha)$ is monotonically increasing in $\alpha$: for $0 < \alpha_1 \le \alpha_2$, we have $\mathcal{D}_g(x, \alpha_1) \le \mathcal{D}_g(x, \alpha_2)$ \citep[Appendix E]{Karimi2016PL}. In our subsequent analysis, this allows us to establish bounds using a localized geometric parameter $\nu_{\mathcal{K}}$ and easily extend them to the global smoothness constant $L$, provided $\nu_{\mathcal{K}} \le L$.

These quantities will allow us to connect convergence rates to geometric properties of the regions visited by the optimization trajectory.

\subsection{The Polyak--{\L}ojasiewicz Condition}

A central question in optimization is to understand when first-order methods achieve \emph{linear convergence}, i.e., when the error decays geometrically: for unconstrained smooth problems with $g=0$,
\begin{equation*}
    f(x_k) - f^* \le (1 - \rho)^k (f(x_0) - f^*)
    \quad\text{
for some $\rho \in (0,1)$}.
\end{equation*}
While strong convexity is a classical sufficient condition for such rates, it is often unnecessarily restrictive in practice.

The \emph{Polyak--{\L}ojasiewicz (PL) condition} provides a strictly weaker alternative that still guarantees linear convergence.

\begin{definition}[Polyak--{\L}ojasiewicz Condition; \citet{Polyak1963Inequality, Lojasiewicz1963topological}]
\label{def:pl}
A differentiable function $f : \mathbb{R}^d \to \mathbb{R}$ satisfies the \emph{PL condition} with parameter $\mu > 0$ if
\begin{equation*}
    \label{eq:pl}
    \frac{1}{2} \|\nabla f(x)\|_2^2 \ge \mu \bigl(f(x) - f^*\bigr)
    \quad \text{for all } x \in \mathbb{R}^d,
\end{equation*}
where $f^* = \inf_x f(x)$.
\end{definition}

The PL inequality links the gradient norm to the suboptimality. Intuitively, it ensures that whenever the objective value is far from optimal, the gradient must be sufficiently large so as to prevent the algorithm from stagnating away from the optimum. Importantly, this condition does not require convexity and can hold for certain nonconvex problems. However, if $f$ is $\alpha$-strongly convex, then it satisfies the PL condition with $\mu = \alpha$. Note that the converse does not hold: the PL condition allows for flat directions and non-unique minimizers, making it significantly more flexible.

When $f$ is $L$-smooth (i.e., has $L$-Lipschitz gradient), gradient descent with step size $1/L$ satisfies
\begin{equation*}
    \label{eq:pl-linear-rate}
    f(x_{k+1}) - f^* \le \left(1 - \frac{\mu}{L}\right)(f(x_k) - f^*),
\end{equation*}
yielding a global linear convergence rate.  However, many problems of interest, including those with constraints or nonsmooth regularizers, fall outside this basic setting. In such cases, the PL condition admits a natural extension to composite objectives of the form
\begin{equation*}
    F(x) = f(x) + g(x),
\end{equation*}
where $g$ may be nonsmooth or encode constraints. In this setting, the gradient norm is replaced by a generalized measure of stationarity based on proximal operators or gradient mappings.

For a composite function $F(x)=f(x)+g(x)$ with an $L$-smooth $f$ and a possibly nonsmooth regularizer $g$, the Polyak--{\L}ojasiewicz condition is extended by the generalized gradient size.
\begin{definition}[Proximal Polyak--{\L}ojasiewicz Condition; \citet{Karimi2016PL}]
\label{def:ppl}
A composite function $F : \mathbb{R}^d \to \mathbb{R}$ satisfies the \emph{proximal PL condition} with parameter $\mu > 0$ if
\begin{equation*}
    \label{eq:ppl}
    \frac{1}{2} \mathcal{D}_g(x, L) \ge \mu \bigl(F(x) - F^*\bigr)
    \quad \text{for all } x \in \mathbb{R}^d,
\end{equation*}
where $F^* = \inf_x F(x)$.
\end{definition}

When $g=0$, the generalized gradient size reduces to
$
    \mathcal{D}_g(x,\alpha)=\|\nabla f(x)\|^2_2,
$
so the proximal PL condition recovers the standard PL condition. This shows that the proximal PL condition is a natural extension of the classical PL condition to composite optimization problems. Thus, our subsequent analysis covers nonsmooth composite problems, such as the LASSO, as well as smooth problems, such as ordinary least squares (OLS) as a special case~\citep[Sec 4.1]{Karimi2016PL}.

While these PL conditions provide a functional characterization of linear convergence, they do not by themselves explain how the associated constants depend on the geometric structure of the problem. In many settings, these constants are governed by underlying geometric properties of the objective. In particular, for problems involving polyhedral structures, such geometric behavior is captured by \emph{Hoffman bounds}, which relates constraint violations to distance from the solution set and will be now discussed.

\subsection{Hoffman Constants: Geometry and Conditioning}
\label{sec:hoffman-background}

A central tool in the analysis of polyhedral systems is the \emph{Hoffman constant} \citep{Hoffman1952ApproximateSolutions}. Given a system of linear constraints, the Hoffman constant quantifies how violations of the constraints translate into distance to feasibility.

\begin{definition}[Hoffman Constant; \citet{Hoffman1952ApproximateSolutions}]
\label{def:hoffman}
For a nonempty polyhedral set
$\mathcal{X} = \{x \in \mathbb{R}^d \mid Gx = h, \; Mx \le r\}$
and given norms $\|\cdot\|_V$ on the domain and $\|\cdot\|_W$ on the residual space, the \emph{Hoffman constant} $H^{V, W}$ is the smallest constant such that
\begin{equation*}
    \dist(x, \mathcal{X})
    :=\min_{x'\in\mathcal{X}}\|x-x'\|_V
    \le H^{V, W} 
    \left\|
    \begin{pmatrix}
    Gx - h \\
    {[Mx - r]}_+
    \end{pmatrix}
    \right\|_W
    \quad \text{for all}\; x \in \mathbb{R}^d,
\end{equation*}
where $[\cdot]_+ = \max(0, \cdot)$ denotes the element-wise positive part. 
\end{definition}

Note that the norms $\|\cdot\|_V$ and $\|\cdot\|_W$ can be different in general. Our subsequent analysis focuses primarily on the Euclidean ($\ell_2$) norm for both spaces, which directly connects the Hoffman constant to the spectral properties of the defining matrices. In this case, we simply denote the Hoffman constant by $H$. 
Since the Hoffman constant measures how far a point is from feasibility given how much it violates the constraints, small $H$ means violations translate efficiently into progress toward feasibility, while large $H$ indicates ill-conditioning.

Hoffman constants play a fundamental role in optimization and numerical analysis. They provide the geometric link between residuals and distance to optimality, and thus directly control convergence rates of first-order methods. In particular, for a wide class of problems, linear convergence is governed by an error contraction factor of $1 -\mathcal{O}(1/(H^2 L))$~\citep{Karimi2016PL}. Beyond optimization, Hoffman-type constants are closely related to classical condition measures in numerical linear algebra. In particular, they can be viewed as generalizations of classical matrix condition numbers and are closely connected to condition measures for linear systems such as Renegar's condition number \citep{renegar2001mathematical} and the Stewart--Todd condition measure \citep{stewart1989scaled,todd1990dantzig}, which quantify sensitivity of feasibility and optimality to perturbations.

Despite their importance, computing Hoffman constants (or closely related error bound constants) is computationally challenging. In general, determining the exact constant for a polyhedral system can be formulated as a nonconvex worst-case optimization problem and is known to be NP-hard in related settings~\citep{pena2000understanding, freund1999some}. Moreover, global Hoffman constants reflect worst-case geometry over the entire ambient space. In high-dimensional settings, this can lead to highly pessimistic bounds, even when the optimization trajectory explores only a small, well-conditioned region.

\section{A Trajectory-Restricted Framework for Optimization}
\label{sec:theory}

Classical convergence analysis relies on \emph{global} regularity conditions which are required to hold over the entire domain. While these conditions guarantee linear convergence, they are often governed by worst-case geometric properties and can therefore be overly pessimistic in high-dimensional or structured problems. Building on the classical PL condition and Hoffman bounds introduced in Section~\ref{sec:background}, 
we now develop a trajectory-restricted framework for analyzing optimization algorithms. 

Rather than requiring conditions to hold globally, we allow them to hold only on subsets of the domain that are relevant to the optimization trajectory. The central idea is to replace global geometric constants with quantities defined on subsets of the domain that are actually explored by the algorithm. This leads to localized versions of classical conditions with constants that reflect the geometry of these subsets. We show that these restricted conditions are sufficient to guarantee linear convergence and that the resulting rates depend explicitly on the geometry of the regions visited by the iterates, providing a simple and flexible framework for capturing the discrepancy between global worst-case bounds and the faster convergence observed in practice.

\subsection{The Restricted Proximal Polyak--{\L}ojasiewicz Inequality}
\label{sec:restricted-pl}

We now introduce a localized version of the proximal Polyak--{\L}ojasiewicz inequality, where instead of requiring a regularity condition to hold everywhere, we only require it to hold on a subset of interest.

\begin{definition}[$\mathcal{K}$-Restricted Proximal PL Inequality]
\label{def:restricted-pl}
Let $\mathcal{K} \subseteq \mathbb{R}^d$. The objective function $F$ satisfies the \emph{$\mathcal{K}$-restricted proximal PL inequality} with constant $\nu_{\mathcal{K}} > 0$ if 
\begin{equation}
\label{eq:restricted-pl}
    \frac{1}{2} \mathcal{D}_g(x, L) \ge \nu_{\mathcal{K}} \big(F(x) - F^*\big) \quad \text{for all } x \in \mathcal{K}\cap \mathop{\mathrm{dom}}(F).
\end{equation}
\end{definition}

When $\mathcal{K} = \mathbb{R}^d$, \eqref{eq:restricted-pl} reduces to the standard (global) proximal PL condition. Thus, the restricted formulation is a direct generalization of classical analysis.

Intuitively, the condition \eqref{eq:restricted-pl} states that within the set $\mathcal{K}$, the generalized gradient size remains sufficiently large relative to the suboptimality gap. In other words, the objective exhibits favorable geometric conditioning on $\mathcal{K}$, even if such structure does not hold globally.

The following result shows that this localized regularity is sufficient to guarantee linear convergence of the proximal gradient method, provided the iterates remain in $\mathcal{K}$.

\begin{theorem}[Restricted Linear Convergence]
\label{thm:restricted-convergence}
For $F$ satisfying Assumption~\ref{ass:regularity}, if the current iterate $x_k$ lies in some $\mathcal{K}$ and $F$ satisfies the $\mathcal{K}$-restricted proximal PL inequality with constant $\nu_{\mathcal{K}}$, then the proximal gradient update with step size $1/L$ satisfies:
\begin{equation*}
    F(x_{k+1}) - F^* \le \left( 1 - \frac{\nu_{\mathcal{K}}}{L} \right) (F(x_k) - F^*).
\end{equation*}
\end{theorem}
\begin{proof}
    We begin with the proximal gradient descent lemma for composite functions with an $L$-smooth component~\citep[Theorem 5]{Karimi2016PL}. 
    For step size $\lambda = 1/L$, we have the bound:
    \begin{equation*}
        F(x_{k+1}) \le F(x_k) - \frac{1}{2L} \mathcal{D}_g(x_k, L).
    \end{equation*}
    Subtracting the optimal value $F^*$ from both sides and invoking the restricted PL inequality 
    yields:
    \begin{align}
        F(x_{k+1}) - F^* &\le F(x_k) - F^* - \frac{1}{L} \left( \frac{1}{2} \mathcal{D}_g(x_k, L) \right) \nonumber \\
        &\le (F(x_k) - F^*) - \frac{\nu_{\mathcal{K}}}{L} (F(x_k) - F^*) \nonumber \\ 
        &= \left(1 - \frac{\nu_{\mathcal{K}}}{L} \right) (F(x_k) - F^*). \nonumber
    \end{align}
    This completes the proof.
\end{proof}

Note that the regularity condition in the proof is only invoked on $\mathcal{K}$. This highlights the important feature of our proposed framework that restricted versions of classical conditions can be introduced straightforwardly and standard convergence arguments carry over with minimal modification, while yielding sharper, geometry-dependent rates.

In practice, the relevant set may vary along the optimization trajectory. For example, in polyhedral problems, the active set may change from one iteration to the next. This leads naturally to a trajectory-dependent notion of convergence rate. 

\begin{corollary}[Trajectory-Dependent Rates]
\label{cor:varying-sets}
    Let $\{x_k\}_{k \ge 0}$ be the sequence of iterates generated by the proximal gradient method with step size $1/L$. 
    Assume that there exists a subset sequence $\{\mathcal{K}_{k}\}_{k\ge 0}$ such that 
    for each iteration $k$,
    the iterate $x_k$ lies in $\mathcal{K}_k$ and $F$ satisfies the $\mathcal{K}_k$-restricted proximal PL inequality with constant $\nu_{\mathcal{K}_k}$.
    Then:
    \begin{equation*}
        F(x_k) - F^* \le \left[ \prod_{i=0}^{k-1} \left( 1 - \frac{\nu_{\mathcal{K}_i}}{L} \right) \right] (F(x_0) - F^*) \quad\text{for all } k \ge 1.
    \end{equation*}
\end{corollary}
\begin{proof}
    The proof follows by recursive application of Theorem~\ref{thm:restricted-convergence}. For any step $i$, since $x_i \in \mathcal{K}_i$, we have:
    \begin{equation*}
        F(x_{i+1}) - F^* \le \left( 1 - \frac{\nu_{\mathcal{K}_i}}{L} \right) (F(x_i) - F^*).
    \end{equation*}
    Unrolling this recurrence relation from $i = k-1$ down to $0$ yields the product bound.
\end{proof}

Because the bound is a product of step-wise contractions, the effective convergence rate over $k$ iterations is governed by the geometric mean of the local rates along the trajectory.

A particularly important consequence arises when the iterates eventually enter and remain within a favorable set, such as an active manifold in a polyhedral problem after a transient phase. In this case, the convergence behavior transitions from a potentially slow transient phase to a faster, geometry-driven regime.

\begin{corollary}[Trajectory Tail Containment]
\label{cor:tail-convergence}
    Let $\mathcal{K} \subseteq \mathbb{R}^d$. Assume the objective $F$ satisfies the $\mathcal{K}$-restricted proximal PL inequality with constant $\nu_{\mathcal{K}}$. If the sequence of iterates enters and remains within $\mathcal{K}$ for all $k \ge T$ (i.e., $x_k \in \mathcal{K}$ for $k \ge T$) and if $\nu_{\mathcal{K}}\le L$, then for all $k \ge T$, the proximal gradient method with step size $1/L$ converges linearly at the restricted rate:
    \begin{equation*}
        F(x_k) - F^* \le \left( 1 - \frac{\nu_{\mathcal{K}}}{L} \right)^{k-T} (F(x_T) - F^*).
    \end{equation*}
\end{corollary}
\begin{proof}
    We apply Corollary~\ref{cor:varying-sets} starting from index $T$. For all $i \in [T, k-1]$, the set is static $\mathcal{K}_i = \mathcal{K}$ and the constant is $\nu_{\mathcal{K}_i} = \nu_{\mathcal{K}}$. The product term $\prod_{i=T}^{k-1} (1 - \nu_{\mathcal{K}_i}/L)$ simplifies to the power $(1 - \nu_{\mathcal{K}}/L)^{k-T}$.
\end{proof}

Note that this result decouples the rate of convergence from the identification of the set~\citep{Sun2019ManifoldIdentification, HareLewis2004Active, Lee2012ManifoldIdentification}. The convergence rate during the transient phase ($k < T$) may be slow (sublinear or governed by a weak global constant), but once the iterates are contained in the set, the faster restricted rate takes over.

\subsection{Restricted Error Bounds and Their Relation to PL}
\label{sec:restricted-eb}

While the proximal PL inequality is convenient for analyzing convergence rates, the geometry of optimization problems is often more naturally expressed in terms of error bounds. We thus introduce a restricted version of the proximal error bound and show that it is related to the restricted PL condition.

\begin{definition}[$\mathcal{K}$-Restricted Proximal Error Bound]
\label{def:restricted-eb}
The objective function $F$ satisfies the $\mathcal{K}$-restricted proximal error bound with constant $\mu_{\mathcal{K}} > 0$ if
\begin{equation*}
    \dist(x, \mathcal{X}^*) \le \mu_{\mathcal{K}} \| \mathcal{G}_{1/L}(x) \|_2 \quad \text{for all} \; x \in \mathcal{K}\cap \mathop{\mathrm{dom}}(F),
\end{equation*}
where $\mathcal{X}^* := \argmin_{x} F(x)$.
\end{definition}

The relationship between proximal error bounds and proximal PL inequalities is well understood in the global setting: For convex objectives, these conditions are known to be equivalent~\citep[Corollary 3.6 and Theorem 5 respectively]{Drusvyatskiy2018ErrorBounds, Bolte2017ErrorBounds}. This equivalence has also been extended to certain nonconvex settings by~\citet[Appendix G]{Karimi2016PL}, although without explicit constants. 

The following result shows the relationship under restriction to a subset $\mathcal{K}$.
It shows the equivalence under additional assumptions and moreover provides explicit relationships between the corresponding constants.

\begin{assumption}[Proximal-Gradient Invariance of $\mathcal{K}$]
\label{assum:stability of proximal gradient update}
For every $\lambda\in (0,1/L]$,
the proximal gradient update $x\mapsto x-\lambda \mathcal{G}_{\lambda}(x)$ with $\lambda$ maps $\mathcal{K}$ into itself.
\end{assumption}
\begin{assumption}[Closedness of $\mathcal{K}$]
    \label{assum:Kclosed}
    $\mathcal{K}$ is closed.
\end{assumption}

\begin{theorem}[Equivalence of Restricted EB and PL]
\label{thm:equivalence}
Let $\mathcal{K} \subseteq \mathbb{R}^d$.
\begin{itemize}
    \item[(i)] \emph{Restricted EB $\implies$ Restricted PL:} If $F$ satisfies the $\mathcal{K}$-restricted proximal EB with constant $\mu_{\mathcal{K}}$, then it satisfies the $\mathcal{K}$-restricted proximal PL inequality with constant
    \begin{equation*}
        \nu_{\mathcal{K}} = \frac{L}{1 + 4L^2 \mu_{\mathcal{K}}^2}.
    \end{equation*}
    \item[(ii)] \emph{Restricted PL $\implies$ Restricted EB:} Conversely, if $F$ satisfies the $\mathcal{K}$-restricted proximal PL inequality with constant $\nu_{\mathcal{K}}$, and in addition, Assumptions \ref{assum:stability of proximal gradient update} and \ref{assum:Kclosed} hold, then it satisfies the $\mathcal{K}$-restricted proximal EB with constant
    \begin{equation*}
        \mu_{\mathcal{K}} = \frac{1}{L} + \frac{2}{\nu_{\mathcal{K}}}.
    \end{equation*}
\end{itemize}
\end{theorem}

This equivalence result shows that the trajectory-restricted framework applies broadly to problems exhibiting error bound behavior. In particular, for any problem class where global error bounds are known to hold, our results immediately yield corresponding restricted PL guarantees and convergence rates. This includes group and sparse group $\ell_1$-regularization~\citep{Tseng2010GroupReg, Zhang2013SparseGroupReg}, nuclear-norm regularization~\citep{Hou2013Nuclear}, and general nonsmooth composite problems~\citep{Zhou2017ErrorBounds, Drusvyatskiy2018ErrorBounds}.

The proof of \Cref{thm:equivalence} is given in Appendix~\ref{app:equivalence} and requires defining a restricted version of the Kurdyka--{\L}ojasiewicz (KL) inequality \citep{Bolte2007KLSubanalytic,bolte2014proximal} and utilizing the subgradient flow. For the restricted concept, we need to investigate the behavior of the subgradient flow within $\mathcal{K}$, which requires approximation using the proximal gradient sequence, rather than the standard proximal-point approximation~\citep{ambrosio2005gradient,peypouquet2010evi},
and path length estimates \citep[e.g.,][]{Bolte2010characterizations,gupta2021path}.

The additional assumption required in the converse direction (Assumption~\ref{assum:stability of proximal gradient update}) bridges the gap between the discrete proximal-gradient iterates and the continuous subgradient flow~\citep{Bolte2007KLSubanalytic} used in the proof. In the global setting, it suffices to assume the existence of a subgradient flow. In the restricted setting, however, we must additionally ensure that the trajectory remains within $\mathcal{K}$ in order to apply the restricted inequality along the flow. This requires an invariance condition on the proximal-gradient iterates, which is formalized in Assumption~\ref{assum:stability of proximal gradient update}.

The connection to the restricted version of the quadratic growth (QG) condition \citep{anitescu2000degenerate,zhang2017restricted,Drusvyatskiy2018ErrorBounds}, i.e.,~the restricted PL implies the restricted QG, is discussed in Appendix \ref{sec:QG}.

In~\cref{ass:regularity}, we have assumed that $f$ is globally $L$-smooth. In parallel with the restricted regularity conditions introduced above, we may instead consider a localized notion of smoothness.

\begin{definition}[$\mathcal{K}$-Restricted $L$-Smoothness]
\label{def:restricted-smoothness}
The \emph{$\mathcal{K}$-restricted smoothness constant} $L_{\mathcal{K}}$ is the smallest constant such that
    \begin{equation*}
        \|\nabla f(x) - \nabla f(y)\|_2 \le L_{\mathcal{K}} \|x - y\|_2 \quad \text{for all } x,y \in \mathcal{K}.
    \end{equation*}
    In this case, we say that $f$ is \emph{$L_{\mathcal{K}}$-smooth} on $\mathcal{K}$.
\end{definition}

This refinement is particularly useful in settings where the objective is poorly conditioned globally but exhibits much stronger regularity on subsets of interest. If the optimization trajectory $\{x_k\}$ remains within $\mathcal{K}$, we can take larger steps of size $1/L_{\mathcal{K}}$, yielding an improved error contraction factor of $(1 - \nu_{\mathcal{K}} / L_{\mathcal{K}})$. Note that this notion differs from the \emph{restricted strong smoothness} (RSS) condition in the literature, which is typically formulated via descent inequalities rather than gradient Lipschitz continuity~\citep{Damadi2022HardThresholding}. 

\subsection{The Restricted Hoffman Constant}

While the trajectory-restricted framework developed so far applies broadly to any objective satisfying local optimization conditions, its practical value lies in deriving explicitly sharper rates by quantifying local geometry. To ground this abstract theory, we now adapt our framework to the important class of polyhedral composite problems. For these problems, geometric conditioning is naturally characterized by the stability of underlying linear constraint systems, which motivates a localized measure of polyhedral regularity.

We now introduce a localized version of the classical Hoffman constant~\citep{Hoffman1952ApproximateSolutions}, which quantifies the relationship between constraint violation and distance to feasibility.

\begin{definition}[Restricted Hoffman Constant]
\label{def:restricted-hoffman}
Let $\mathcal{X} \subseteq \mathbb{R}^{d}$ be a nonempty polyhedral set 
\begin{equation*}
    \mathcal{X} = \{x \in \mathbb{R}^d \mid Gx = h, \; Mx \le r\}.
\end{equation*}
Given a subset $\mathcal{K} \subseteq \mathbb{R}^d$ and norms $\|\cdot\|_V$ and $\|\cdot\|_W$, the \emph{$\mathcal{K}$-restricted Hoffman constant}, denoted $H_{\mathcal{K}}$, is the smallest constant such that
\begin{equation*}
    \dist(x, \mathcal{X}) :=\min_{x'\in\mathcal{X}}\|x-x'\|_V \;\le\; H^{V, W}_{\mathcal{K}} 
    \left\| 
    \begin{pmatrix} 
    Gx - h \\ 
    {[Mx - r]}_+ 
    \end{pmatrix} 
    \right\|_W
    \quad \text{for all}\; x \in \mathcal{K}.
\end{equation*}
\end{definition}

There are three important observations to note. First, the Hoffman constant $H$ is often studied assuming a Euclidean norm $\ell_2$, in which case $H$ is directly related to the spectral properties of the system of constraints; over the subset $\mathcal{K}$, we denote this simply by $H_{\mathcal{K}}$ for brevity. Second, the definition recovers the classical (global) Hoffman constant when $\mathcal{K} = \mathbb{R}^d$. Third, for linear systems, the Hoffman constant is directly related to the singular values of the defining matrix and is always finite when the system is well-posed (see e.g.,~\citet[Theorem 4.7]{Leventhal2008Randomized} and~\citet[Appendix F]{Karimi2016PL}). The following result illustrates this connection in the simplest setting of linear systems.

\begin{proposition}[Hoffman Constant of Linear Systems; e.g.,~\citep{Leventhal2008Randomized}]
\label{prop:hoffman-singular-value}
    Consider the linear system $\mathcal{X} = \{x \mid Gx = h\}$. The global Hoffman constant $H_{\mathbb{R}^d}$ (in the $\ell_2$-norm) is the reciprocal of the smallest nonzero singular value of $G$:
    \begin{equation*}
        H_{\mathbb{R}^d} = \|G^{\dagger}\|_2 = 1/\sigma_{\min}^+(G) = 1\Big{/}\sqrt{\lambda_{\min}^+ (G^\top G)},
    \end{equation*}
    where $G^{\dagger}$ is the pseudoinverse of $G$ with $\sigma_{\min}^{+}(G)$ the smallest nonzero singular value of $G$, and $G^\top G$ is the Gram matrix with $\lambda_{\min}^{+}(G^{\top}G)$ its smallest nonzero eigenvalue.
\end{proposition}

This characterization highlights an important distinction between linear systems and composite polyhedral optimization problems. For a purely linear system defined by a sub-Gaussian random matrix $A \in \mathbb{R}^{n \times d}$ with $d \gg n$, the smallest nonzero singular value typically scales as $\sigma_{\min}^+(A) \sim \sqrt{d} - \sqrt{n}$~\citep{Rudelson2010ExtremeSingularValues}. Consequently, the global Hoffman constant for the linear constraint shrinks in high dimensions, yielding $H \sim 1/\sqrt{d}$. In this isolated context, the underlying linear subspace becomes increasingly well-conditioned as the ambient dimension grows.

However, for general polyhedral composite problems, the geometry is governed by a full constraint system that couples the linear map with ambient inequality constraints. To bound the global distance to optimality, the Hoffman constant must account for the worst-case geometric distortion across all of these polyhedral faces. If the constraints are dense, meaning the normal vectors defining the faces span all ambient dimensions, translating their residuals into a standard Euclidean distance incurs a severe dimensional penalty. As we will explicitly demonstrate for sparse regression problems further on in Section~\ref{sec:lasso-svm-applications}, this ambient distortion forces the global Hoffman constant to degrade as $H \sim \sqrt{d}$.

These observations motivate our localized approach: by evaluating the Hoffman constant strictly on the structured subsets actually encountered by the optimization trajectory, we can bypass this ambient distortion entirely. In the following sections, we establish how this restricted constant formally governs local convergence rates, and subsequently apply it to resolve the high-dimensional scaling of these polyhedral regularizers promoting sparsity.

\subsection{Strongly Convex Losses with Linear Structure}

We consider composite problems of the form
\begin{equation*}
    F(x) = f(Ax) + g(x),
\end{equation*}
where $A \in \mathbb{R}^{n \times d}$ is a data or design matrix. Building on our standing assumptions, we further assume here that $f: \mathbb{R}^n \to \mathbb{R}$ is strongly convex, and $g: \mathbb{R}^d \to \mathbb{R} \cup \{+\infty\}$ encodes polyhedral structure. This class includes constrained optimization as well as many standard models in machine learning, such as (regularized) linear regression, logistic regression, 
and support vector machines. The KKT condition in this setting becomes $0 \in A^\top \nabla f(Ax^*) + \partial g(x^*)$.

\paragraph{Polyhedral Constraints.} To develop restricted conditions for this general class, we first examine the case where $g$ represents a hard polyhedral constraint, i.e., an indicator function $g(x) := \delta_{\mathcal{X}}(x)$ over a polyhedral set $\mathcal{X}$. While global PL conditions for this specific constrained setting are known~\citep[Appendix F, Problem 3]{Karimi2016PL}, the corresponding constants depend on global Hoffman bounds and may be highly conservative. We show that, under restriction to a subset $\mathcal{K}$, the PL constant instead depends on a restricted Hoffman constant, leading to significantly sharper rates. The following result makes this connection precise. We note that the condition that $\mathcal{K} \subseteq \mathcal{X}$ is not restrictive in practice; for instance, proximal gradient iterates remain feasible and therefore lie within $\mathcal{X}$.

\begin{proposition}[Restricted PL for Indicator Functions]
\label{prop:indicator-pl}
    Consider the composite problem $F(x) = f(Ax) + \delta_{\mathcal{X}}(x)$, where $f$ is $\alpha$-strongly convex and $\mathcal{X} \subseteq \mathbb{R}^d$ is a nonempty polyhedral set $\mathcal{X}=\{x\in \mathbb{R}^{d} \mid Gx\le h\}$. Let $\mathcal{K} \subseteq \mathcal{X}$ be a subset of the domain $\mathcal{X}$ and $H_{\mathcal{K}}$ be the restricted Hoffman constant such that 
    \begin{equation*}
        \dist(x, \mathcal{X}^*) \le H_{\mathcal{K}} \| A(x - x^*) \|_2
        \quad
        \text{for all $x \in \mathcal{K}$},
    \end{equation*}
    where $x^* = \mathcal{P}_{\mathcal{X}^*}(x)$ is the projection onto the optimal set $\mathcal{X}^* :=\argmin_{x\in\mathcal{X}}f(Ax)$. 
    Let $\nu_{\mathcal{K}} = \alpha/H_{\mathcal{K}}^2$.
    If $\nu_{\mathcal{K}}\le L$, then
    $F$ restricted to $\mathcal{K}$ satisfies the $\mathcal{K}$-restricted proximal PL inequality with constant $\nu_{\mathcal{K}}$.
\end{proposition}
\begin{proof}
    Fix $x \in \mathcal{K}$ arbitrarily. Since $\mathcal{K} \subseteq \mathcal{X}$, $x$ is feasible, so the indicator term vanishes $\delta_{\mathcal{X}}(x) = 0$ and $F(x) = f(Ax)$. Let $x^* = \mathcal{P}_{\mathcal{X}^*}(x)$ be the projection of $x$ onto the optimal set.
    
    Define $\tilde{f}(x) := f(Ax)$ such that $\nabla \tilde f(x) = A^\top \nabla f(Ax)$. By the $\alpha$-strong convexity of $f$ on the range of $A$:
    \begin{equation}
        \label{eq:sc_expansion.primitive}
        \tilde{f}(x^*) \ge \tilde{f}(x) + \langle \nabla \tilde{f}(x), x^* - x \rangle + \frac{\alpha}{2} \| A(x^* - x) \|_2^2.
    \end{equation}
    Since $x^* \in \mathcal{X}^* \subseteq \mathcal{X}$, both points $x,x^{*}$ are feasible. Hence inequality \eqref{eq:sc_expansion.primitive} provides us with an expression:
    \begin{equation}
        \label{eq:sc_expansion}
        F^* \ge F(x) + \langle \nabla \tilde{f}(x), x^* - x \rangle + \frac{\alpha}{2} \| A(x - x^*) \|_2^2 + g(x^*) - g(x).
    \end{equation}
    
    Now, we derive the Hoffman bound for the system. The optimal set $\mathcal{X}^*$ is defined by the intersection of the feasible set and the minimizers of the strongly convex loss. Since $f$ is strongly convex, $z^* :=Ax^*$ is unique. Thus, $\mathcal{X}^*$ is defined by the polyhedral system
    \begin{equation*}
        \mathcal{X}^* = 
        \left\{ y \in \mathbb{R}^d \; \middle| \; A y = z^*, \; G y \le h \right\}.
    \end{equation*}
    To apply the Hoffman bound, we express the equality $Ay = z^*$ as pairs of inequalities, and then the system defining the distance to optimality becomes
    \begin{equation*}
        \mathbf{P}y \le \mathbf{q} \quad \text{where} \quad 
        \mathbf{P} = \begin{pmatrix} A \\ -A \\ G \end{pmatrix}, \quad 
        \mathbf{q} = \begin{pmatrix} z^* \\ -z^* \\ h \end{pmatrix}.
    \end{equation*}
    Since $x \in \mathcal{K} \subseteq \mathcal{X}$, the constraint $Gx \le h$ is satisfied, so $(Gx - h)_+ = 0$. Using the identity $a^2 = (a_+)^2 + (-a_+)^2$, the norm of the residual $[\mathbf{P}x - \mathbf{q}]_+$ simplifies to
    \begin{equation*}
     \|[\mathbf{P}x - \mathbf{q}]_+\|_{2}=
     \| Ax - z^* \|_2 =  \| A(x - x^*) \|_2.
    \end{equation*}
    Hence the $\mathcal{K}$-restricted Hoffman constant $H_{\mathcal{K}}$ yields
    \begin{equation*}
    \|x - x^*\|_2 
    = \dist(x, \mathcal{X}^{*})
    \le H_{\mathcal{K}} \| A(x - x^*) \|_2.
    \end{equation*}
    Squaring and rearranging this yields the bound on the geometry of the map $A$,
    \begin{equation*}
        \label{eq:hoffman_substitution}
        \| A(x - x^*) \|_2^2 \ge \frac{1}{H_{\mathcal{K}}^2} \| x - x^* \|_2^2.
    \end{equation*}
    Substituting into \eqref{eq:sc_expansion} leads to
    \begin{equation}
        \label{eq:hoffman_substituted}
        F^* \ge F(x) + \langle \nabla \tilde{f}(x), x^* - x \rangle + \frac{\alpha}{2 H_{\mathcal{K}}^2} \| x - x^* \|_2^2.
    \end{equation}

    Let $\nu_{\mathcal{K}} = \alpha / H_{\mathcal{K}}^2$. The RHS of \eqref{eq:hoffman_substituted} is bounded by the minimization problem that defines $\mathcal{D}_g$. Specifically, since $x^*$ is a feasible point ($g(x^*) < \infty$), the minimum over all $y$ is bounded above by the value at $y=x^*$:
    \begin{align*}
        \min_y \left[ \langle \nabla \tilde{f}(x), y - x \rangle + \frac{\nu_{\mathcal{K}}}{2} \| y - x \|_2^2 + g(y) - g(x) \right] 
        &\le \langle \nabla \tilde{f}(x), x^* - x \rangle + \frac{\nu_{\mathcal{K}}}{2} \| x^* - x \|_2^2 + \underbrace{g(x^*) - g(x)}_{0} \\
        &= \langle \nabla \tilde{f}(x), x^* - x \rangle + \frac{\alpha}{2 H_{\mathcal{K}}^2} \| x^* - x \|_2^2.
    \end{align*}
    Substituting this upper bound back into \eqref{eq:hoffman_substituted} gives
    \begin{equation*}
        F^* \ge F(x) + \min_y \left[ \langle \nabla \tilde{f}(x), y - x \rangle + \frac{\nu_{\mathcal{K}}}{2} \| y - x \|_2^2 + g(y) - g(x) \right].
    \end{equation*}
    By definition, the minimization term is exactly $-(1/2\nu_{\mathcal{K}}) \mathcal{D}_g(x, \nu_{\mathcal{K}})$. Thus,
    \begin{equation*} 
        F^* \ge F(x) - \frac{1}{2\nu_{\mathcal{K}}} \mathcal{D}_g(x, \nu_{\mathcal{K}}).
    \end{equation*}
    Rearranging terms yields the restricted proximal PL inequality:
    \begin{equation*}
        \frac{1}{2} \mathcal{D}_g(x, \nu_{\mathcal{K}}) \ge \nu_{\mathcal{K}} (F(x) - F^*) \quad \text{with } \nu_{\mathcal{K}} = \frac{\alpha}{H_{\mathcal{K}}^2}.
    \end{equation*}
    Lastly, by the monotonic increasing property of  $\mathcal{D}_{g}(x,\alpha)$ with respect to $\alpha$, we obtain
    \[
    \frac{1}{2}\mathcal{D}_g(x, L) \ge \nu_{\mathcal{K}} (F(x) - F^*) \quad \text{with } \nu_{\mathcal{K}} = \frac{\alpha}{H_{\mathcal{K}}^2},
    \]
    which completes the proof.
\end{proof}

This result shows that the restricted PL constant is directly controlled by the restricted Hoffman constant, with $\nu_{\mathcal{K}} = \alpha / H_{\mathcal{K}}^2$. In particular, improved geometric conditioning of the active region translates immediately into faster convergence rates; see Section~\ref{sec:lasso-svm-applications} for details. 

\paragraph{Polyhedral Convex Regularizers.} We next extend this result to finite-valued polyhedral convex regularizers $g(x)$. 
Consider the composite problem $F(x) = f(Ax) + g(x)$, where $f$ is $\alpha$-strongly convex and $L$-smooth,
and $g(x):\mathbb{R}^d \to\mathbb{R}$ is a finite-valued polyhedral convex function.
Let $\mathcal X^*:=\arg\min_{x\in\mathbb R^d}\{f(Ax)+g(x)\}$
denote the optimal set.

\begin{lemma}
    The optimal set is polyhedral.
\end{lemma}
\begin{proof}
Since $g$ is polyhedral, there exist
vectors $a_1,\dots,a_m\in\mathbb R^d$ and scalars $b_1,\dots,b_m\in\mathbb R$ such that
\[
g(x)=\max_{1\le i\le m}\{a_i^\top x+b_i\}
\qquad
\text{for all }x\in\mathbb R^d.
\]
Since $f$ is strongly convex, the image $Ax$ is constant over $\mathcal X^*$. Therefore, there exists a unique vector $y^*$ such that
$Ax=y^*$ for all $x\in\mathcal X^*$.
Define $\xi^*:=\nabla f(y^*)$ and fix any $x^*\in\mathcal X^*$. By the first-order optimality condition, we have
\[
0\in A^\top \nabla f(Ax^*)+\partial g(x^*)
=
A^\top \xi^*+\partial g(x^*).
\]
Since the condition $-A^{\top}\xi^{*} \in \partial g(x)$ is equivalent to
\[
x \in M[\xi^{*}]:=\argmin_{u}\{ g(u) + \langle A^{\top}\xi^{*}, u \rangle \},
\] 
there exists a matrix $B$ and a vector $c$ \citep[Corollary 11.16]{Rockafellar1998Variational} such that
\begin{align}
\label{eq: defining B and c in polyhedron}
M[\xi^{*}]=\{x\in\mathbb R^d \mid -A^\top\xi^*\in\partial g(x)\}
=
\{x\in\mathbb R^d\mid Bx\le c\}.
\end{align}
Therefore, we obtain the polyhedral representation
$\mathcal X^* = \{x\in\mathbb R^d \mid Ax=y^*,\ Bx\le c\}.$ Writing the equality $Ax=y^*$ as two inequalities, we obtain $\mathcal X^* = \{x\in\mathbb R^d:\bar Mx\le \bar q\},$
where
\[
\bar M:=
\begin{pmatrix}
A\\
-A\\
B
\end{pmatrix},
\qquad
\bar q:=
\begin{pmatrix}
y^*\\
-y^*\\
c
\end{pmatrix}.
\]
This completes the proof.
\end{proof}

Having established that the optimal set $\mathcal X^*$ is polyhedral, we can now 
invoke Hoffman-type bounds to control the distance to optimality. In line with 
our trajectory-restricted framework, we consider such bounds restricted to a set 
$\mathcal{K} \subseteq \mathbb{R}^d$ that contains $\mathcal{X}^*$ and captures 
the region explored by the optimization algorithm.

Let $H_{\mathcal{K}}$ denote the restricted Hoffman constant associated with the 
polyhedral system defining $\mathcal{X}^*$, i.e.,
\[
\dist(x,\mathcal{X}^{*}) \le H_{\mathcal{K}}\{ \|Ax - y^{*}\|_{2} + \|(Bx-c)_{+}\|_{2} \}
\quad \text{for all $x\in\mathcal{K}$}.
\]

In contrast to the indicator-function setting considered earlier, the presence of 
a finite-valued polyhedral regularizer $g$ requires an additional curvature-type 
condition to relate function values to distance from the optimal set. To this end, 
we introduce a restricted notion of firm convexity, following~\citet{Drusvyatskiy2018ErrorBounds}.


\begin{definition}[Restricted Firm Convexity]
\label{def: restricted firm convexity constant}
For $\mathcal{K}$, the restricted firm convexity constant of \(g\) relative to \(-A^\top \xi^*\) is the largest constant \(\gamma_{\mathcal{K}}>0\) such that
\[
g(x)+\langle A^\top \xi^*,x\rangle
-
\inf_{u\in M[\xi^*]}
\bigl(g(u)+\langle A^\top \xi^*,u\rangle\bigr)
\ge
\frac{\gamma_{\mathcal{K}}}{2}\,
\dist\bigl(x,M[\xi^*]\bigr)^2
\qquad
\text{for all $x\in\mathcal{K}$}.
\]
\end{definition}


The following result makes the connection between the restricted geometric constants 
and error bound guarantees for composite problems with polyhedral convex regularizers.

\begin{proposition}[Restricted PL for Polyhedral Nonsmooth Functions]
\label{prop:nonsmooth-polyhedral-pl}
    Denote by $H_{\mathcal{K}}$ and $\gamma_{\mathcal{K}}$ the restricted Hoffman and firm convexity constants respectively.
    Let $\|B\|$ be the operator norm of $B$ defined in \eqref{eq: defining B and c in polyhedron}.
    Assume that the proximal gradient update from a point in $\mathcal{K}$ remains in $\mathcal{K}$.
    Then $F$ satisfies the restricted EB inequality with a constant $L^{-1}
+8\alpha^{-1} H_{\mathcal{K}}^{2}\bigl(1+\|B\|\gamma^{-1}_{\mathcal{K}}L\bigr)^{2}
+8 H_{\mathcal{K}} \|B\| \gamma^{-1}_{\mathcal{K}}$.
\end{proposition}


The proof is given in Appendix~\ref{app:nonsmooth-polyhedral-g:rev}. The key takeaway is that the error bound, and hence the convergence rate via the equivalence results of Section~\ref{sec:restricted-eb}, is governed by two geometric quantities: the restricted Hoffman constant $H_{\mathcal{K}}$, which captures the stability of the polyhedral optimal set, and the restricted firm convexity constant $\gamma_{\mathcal{K}}$, which quantifies the local curvature of the regularizer.

In particular, when the trajectory enters a well-conditioned region $\mathcal{K}$, both constants can improve significantly compared to their global counterparts, leading to faster convergence rates. This shows that for polyhedral composite problems, the effective conditioning of the optimization problem is determined by the local geometry of the system, as encoded by these restricted constants.

\subsection{A Conditioning Perspective on Convergence}
\label{subsec:conditioning-perspective}

We now reinterpret the preceding results through the lens of conditioning, linking the restricted geometric constants to classical quantities such as Lipschitz constants and condition numbers that govern convergence rates in practice.

The global variant of the linear convergence guarantees established in \cref{thm:restricted-convergence} assumes the form 
\begin{equation*}
    F(x_{k+1}) - F^* \le \left(1 - \frac{\nu}{L}\right)(F(x_k) - F^*),
\end{equation*}
where $L$ is the global smoothness constant of $f$ and $\nu$ is the proximal PL constant. In contrast, the restricted geometric variant established in \cref{thm:restricted-convergence} yields a localized convergence bound:
\begin{equation*}
    F(x_{k+1}) - F^* \le \left(1 - \frac{\nu_{\mathcal{K}}}{L_{\mathcal{K}}}\right)(F(x_k) - F^*),
\end{equation*}
where $L_{\mathcal{K}}$ and $\nu_{\mathcal{K}}$ are the restricted versions of the smoothness constant and proximal PL constant respectively. The ratio of these quantities naturally defines the effective \emph{condition numbers} for the global and restricted settings:
\begin{equation*}
    \kappa := \frac{L}{\nu}, \quad \text{and} \quad \kappa_{\mathcal{K}} := \frac{L_{\mathcal{K}}}{\nu_{\mathcal{K}}}.
\end{equation*}
These condition numbers govern the speed of convergence: a smaller condition number yields a faster linear rate.

The constant $L$ arises from the assumption that the gradient of $f$ is Lipschitz continuous:
\begin{equation*}
    \|\nabla f(x) - \nabla f(y)\|_2 \le L \|x - y\|_2.
\end{equation*}
Intuitively, $L$ measures the curvature of the objective: it controls how quickly the gradient can change and therefore limits how large a step an optimization algorithm can practically take. For quadratic losses of the form
\begin{equation*}
    f(x) = \|Ax - y\|_2^2 / 2,
\end{equation*}
the smoothness constant is given by
\begin{equation*}
    L = \lambda_{\max}(A^\top A) = \sigma_{\max}^2(A),
\end{equation*}
where $A^\top A$ is the \emph{Gram matrix}. Thus, $L$ is determined by the largest singular value of the data matrix $A$. As is well known in random matrix theory, for an $n \times d$ sub-Gaussian random design matrix with $d \gg n$, this global spectral penalty scales as $L \sim (\sqrt{d} + \sqrt{n})^2 \approx d$~\citep{Rudelson2010ExtremeSingularValues}.

On the other hand, the PL constant $\nu_{\mathcal{K}}$ is controlled by the geometry of the constraint structure. In the polyhedral setting, it satisfies
\begin{equation*}
    \nu_{\mathcal{K}} \;\propto\; \frac{1}{H_{\mathcal{K}}^2},
\end{equation*}
where $H_{\mathcal{K}}$ is the (restricted) Hoffman constant. This links convergence directly to the conditioning of the underlying polyhedral system.

Combining these observations offers a unifying framework where the effective condition number governing convergence can be expressed as
\begin{equation*}
    \kappa_{\mathcal{K}} \;\propto\; L_{\mathcal{K}} \cdot H_{\mathcal{K}}^2.
\end{equation*}
This decomposition highlights two complementary sources of difficulty: (i) Spectral geometry (via $L_{\mathcal{K}}$): determined by the data matrix $A$ and its singular values; and (ii) Polyhedral geometry (via $H_{\mathcal{K}}$): determined by the structure of the constraints or regularizer. In particular, $L_{\mathcal{K}}$ reflects the largest eigenvalue of the (restricted) Gram matrix, while $H_{\mathcal{K}}$ is inversely related to its smallest eigenvalue, so that together they recover the standard condition number governing convergence.

In high-dimensional problems, both quantities may be poorly behaved globally. However, when the optimization trajectory is restricted to a structured subset $\mathcal{K}$, both $L_{\mathcal{K}}$ and $H_{\mathcal{K}}$ can improve dramatically. This yields significantly better local conditioning and explains the fast convergence observed in practice. This phenomenon appears in the following two canonical settings that we study next.

This conditioning perspective suggests that improving convergence is fundamentally a problem of improving geometry, either by reducing curvature (through $L_{\mathcal{K}}$) or exploiting structure of the polyhedral geometry (through $H_{\mathcal{K}}$).

\section{Case Studies and Extensions}
\label{sec:lasso-svm-applications}

We illustrate the trajectory-restricted framework through two canonical models: the LASSO and support vector machines (SVM). Both fit the composite form $F(x) = f(Ax) + g(x)$ with polyhedral $g$, so their convergence behavior is governed by restricted Hoffman constants as we have previously established in~\cref{prop:nonsmooth-polyhedral-pl}. 

These examples highlight a central phenomenon: although global conditioning may be poor in high-dimensional settings, optimization algorithms typically traverse much smaller, structured regions of the space where the effective geometry is significantly more favorable. This viewpoint not only explains the observed two-phase convergence behavior, but also motivates broader extensions of the framework to a larger class of algorithms.

\subsection{Sparse Regression: LASSO}

Consider the $\ell_1$-regularized least squares problem, which is known as LASSO~\citep{Tibshirani1996Lasso}:
\begin{equation*}
    \min_{\beta \in \mathbb{R}^d} F(\beta):=\frac{1}{2}\|A \beta - y\|_2^2 + \eta \|\beta\|_1,
\end{equation*}
where $A \in \mathbb{R}^{n \times d}$ is the design matrix, $\beta \in \mathbb{R}^d$ are the regression coefficients, $y \in \mathbb{R}^n$ is the target vector, and $\eta$ is the regularization parameter. This model promotes sparse solutions and is particularly relevant in high-dimensional settings ($d \gg n$), where the ordinary least squares OLS estimator is ill-posed. The corresponding proximal method is the iterative shrinkage-thresholding algorithm (ISTA).

We now consider the application of Proposition \ref{prop:nonsmooth-polyhedral-pl} to the LASSO problem. We first provide a polyhedral representation of $M[\hat{\xi}]$,
then bound the coefficient matrix $B$ in \eqref{eq: defining B and c in polyhedron}
and the firm convexity constant $\gamma_{\mathcal{K}}$.
Let $\hat{\beta}$ be an arbitrary LASSO solution and define
\[
\hat{y}:=A\hat{\beta},
\qquad
\hat{\xi}:=\nabla f(\hat{y})=\hat{y}-y,
\qquad
\hat{s}:=-A^\top \hat{\xi}=A^\top (y-A\hat{\beta}).
\]
The first-order optimality condition for LASSO is
$0\in A^\top(A\hat{\beta}-y)+\eta\,\partial\|\hat{\beta}\|_1$,
and hence we have, for each $i=1,\dots,d$,
\[
\hat{s}_i=
\begin{cases}
\eta \cdot \mathrm{sign}(\hat{\beta}_i), & \hat{\beta}_i\neq 0,\\[1mm]
\in[-\eta,\eta], & \hat{\beta}_i=0.
\end{cases}
\]
We then introduce the index sets separating the sign of $\hat{s}_{i}$:
\[
I_0:=\{i\mid\ |\hat{s}_i|<\eta\},
\qquad
I_+:=\{i\mid\ \hat{s}_i=\eta\},
\qquad
I_-:=\{i\mid\ \hat{s}_i=-\eta\}.
\]
With this notation, the set $M[\hat{\xi}]$ defined in \eqref{eq: defining B and c in polyhedron},
$M[\hat{\xi}]:=
\{\beta\in\mathbb{R}^d\mid \ -A^\top \hat{\xi}\in \eta\,\partial\|\beta\|_1\}$, admits the explicit form
\[
M[\hat{\xi}]
=
\Bigl\{
\beta\in\mathbb{R}^d \mid 
\beta_i=0 \ \text{for } i\in I_0,\ \
\beta_i\ge 0 \ \text{for } i\in I_+,\ \
\beta_i\le 0 \ \text{for } i\in I_-
\Bigr\}.
\]
Accordingly, using the standard basis vectors $\{e_i\}$, we can represent the polyhedral system as
\[
M[\hat{\xi}]=\{\beta\in\mathbb R^d\mid\ B\beta\le 0\}
\quad\text{with}\quad
B=
\begin{pmatrix}
(e_i^\top)_{i\in I_0}\\
(-e_i^\top)_{i\in I_0}\\
(-e_i^\top)_{i\in I_+}\\
(e_i^\top)_{i\in I_-}
\end{pmatrix}.
\]
Hence, we have $\|B\|\le \sqrt{2}$.

To obtain an explicit bound on the restricted firm convexity constant 
(Definition~\ref{def: restricted firm convexity constant}), we introduce 
a standard nondegeneracy measure.

\begin{definition}[Strict Complementarity Margin]
\label{def:strict-complementarity}
The \emph{strict complementarity margin} is defined as
\[
\delta_* := \min_{i \in I_0} (\eta - |\hat{s}_i|),
\]
which quantifies the separation between active and inactive coordinates at the solution.
\end{definition}

To control the constants globally along the optimization trajectory, we restrict attention 
to a bounded region. Using the coercive lower bound $F(\beta) \ge \eta \|\beta\|_1$, 
we define the level set
\[
\Theta := \left\{ \beta \in \mathbb{R}^d \;\middle|\; \|\beta\|_{\infty} \le \frac{F(\beta_0)}{\eta} \right\},
\]
which contains all iterates initialized at $\beta_0$.


For any $\mathcal{K} \subseteq \Theta$, the restricted firm convexity constant admits the bound
\[
\frac{2}{\gamma_{\mathcal K}}
\le
\max\left\{
\frac{F(\beta^0)}{\eta\,\delta_*},\,
\frac{F(\beta^0)}{2\eta^2}
\right\}.
\]
The proof is given in Appendix~\ref{app:Restricted metric subregulairy for LASSO}.


This bound shows that $\gamma_{\mathcal K}$ is controlled by the strict complementarity margin 
and the initial level set. Consequently, combining this with Proposition~\ref{prop:nonsmooth-polyhedral-pl}, 
we see that the convergence rate is governed by the restricted Hoffman constant $H_{\mathcal K}$, 
up to problem-dependent constants that remain bounded along the trajectory.

For simplicity, going forwards, we assume that $\mathcal{K} \subseteq \Theta$ and omit explicit dependence on $\Theta$.



\paragraph{Active Sets and Local Geometry.} Define the active set $\mathcal{A}(\beta) = \{j : \beta_j \neq 0\}$ and let $\hat{\mathcal{A}} = \operatorname{supp}(\hat{\beta})$ denote the support of a minimizer $\hat{\beta}$. Intuitively, this is the sparse dictionary of features used for the regression.

This combinatorial structure induces a natural decomposition of the space and determines the local geometry of the problem. As anticipated in Section~\ref{subsec:conditioning-perspective}, the global geometric conditioning of the LASSO is severely penalized by the ambient dimension. The polyhedral constraints of the LASSO arise from the $\ell_1$ penalty, which geometrically corresponds to an ambient cross-polytope between the inactive and active features. The normal vectors defining the faces of this cross-polytope are dense vectors of the form $[\pm 1, \dots, \pm 1]^\top \in \mathbb{R}^d$. Notably, the standard Euclidean ($\ell_2$) length of these normal vectors is exactly $\sqrt{d}$. Because the global Hoffman constant must translate KKT residuals associated with these dense faces into an ambient $\ell_2$ distance, it fundamentally incurs this worst-case geometric distortion. Consequently, the global Hoffman constant for the LASSO inherently degrades as $H \sim \sqrt{d}$, as observed in \cref{fig:hoffman-scaling}.

\begin{figure}[ht]
    \centering
    \includegraphics[width=0.5\linewidth]{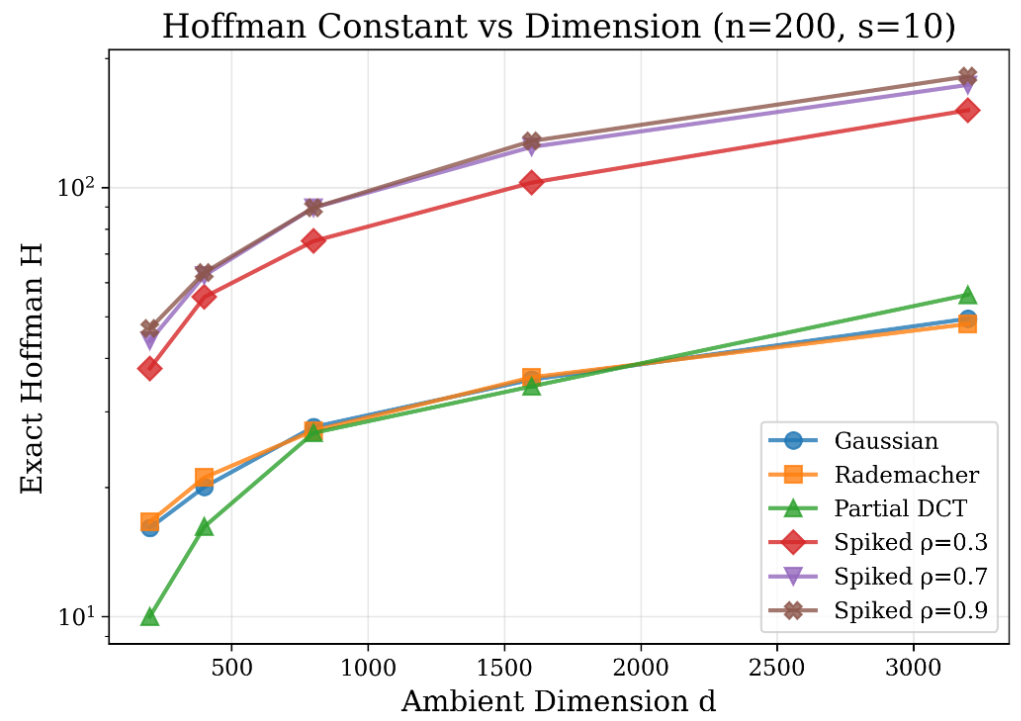}
    \caption[Scaling of the Hoffman constant with ambient dimension.]{\textbf{Scaling of the Hoffman constant with ambient dimension ($n = 200, s = 10$).} The exact Hoffman constant $H$ of the LASSO is computed for a fixed, correctly identified active set across varying ambient dimensions $d$. We observe that the overall Hoffman constant increases at a rate of roughly $\mathcal{O}(\sqrt{d})$ for all ensembles. Highly correlated designs (spiked models) exhibit strictly worse geometric conditioning than standard independent ensembles (Gaussian, Rademacher).}
    \label{fig:hoffman-scaling}
\end{figure}

A key advantage of our trajectory-restricted framework is that the subset $\mathcal{K}$ can be chosen to reflect the active structure and bypass this ambient distortion. We highlight three canonical choices:

\begin{enumerate}
    \item \textbf{Global, $\mathcal{K} = \mathbb{R}^d$:} This corresponds to the standard analysis~\citep[Section 3.2.1]{Bolte2017ErrorBounds}.
    The global Hoffman constant $H_{\mathbb{R}^d}$ bounds the worst-case geometry over the entire domain and thus scales as $H_{\mathbb{R}^d} \sim \sqrt{d}$. 
    
    \item \textbf{Local, $\mathcal{K} = \mathcal{S}_{\hat{\mathcal{A}}}$:} Let $\hat{\mathcal{A}} = \operatorname{supp}(\hat{\beta})$ be the optimal active set and $\mathcal{S}_{\hat{\mathcal{A}}}:=\{\beta \in \mathbb{R}^{d}\mid \operatorname{supp}(\beta)=\hat{\mathcal{A}}\}$. Restricting the geometry to this optimal support yields a restricted Hoffman constant $H_{\hat{\mathcal{A}}}$ that depends only on the active submatrix. Because the effective dimension of this face is $s = |\hat{\mathcal{A}}|$, the dense $\sqrt{d}$ penalty is bypassed entirely since $A_{\hat{\mathcal{A}}} \in \mathbb{R}^{n \times s}$. 
    Under standard assumptions on $y$ and $A$ \citep[Corollary 7.22]{Wainwright2019HighDimensional} and under the assumption that the true (noiseless) regression coefficient $\beta^{*}$ is $s$-sparse, 
    $\hat{\mathcal{A}}$ is contained within $\operatorname{supp}(\beta^{*})$
    with a high probability;
    hence this restricted constant $H_{\hat{\mathcal{A}}}$ scales only with $\sqrt{n}$ and is independent of the ambient dimension $d$; see also ~\citep{Zhao2006ModelSelectionConsistency}.
    
    \item \textbf{Trajectory-Dependent, $\mathcal{K} = \cup_k \mathcal{A}_k$:} During optimization, the active set evolves and since the algorithm visits a trajectory of active sets $\mathcal{T} = \{\mathcal{A}_0, \mathcal{A}_1, \dots\}$, the convergence at iteration $k$ is governed by the geometry of the current subproblem; if the algorithm is initialized well, then this may mean that the trajectory is well-conditioned throughout. This leads to a path-dependent rate: the convergence rate at iteration $k$ is governed by the geometry of the current subproblem and the overall path-dependent rate is controlled by
    \begin{equation*}
        H_{\mathcal{T}} = \max_{\mathcal{A} \in \mathcal{T}} H_{\mathcal{A}}.
    \end{equation*}
\end{enumerate}

\paragraph{Two-Phase Convergence.} This viewpoint explains the characteristic two-phase behavior of proximal methods:

\begin{itemize}
    \item \textbf{Phase 1: Manifold Identification (Search Phase).} During early iterations, the active set evolves. Convergence is governed by trajectory-dependent constants and may be slow. Algorithms such as ISTA possess the \emph{finite manifold identification property}~\citep{Sun2019ManifoldIdentification, Nutini2019ActiveSet} under the  strict complementarity condition: $\delta_{*}>0$. During the early transient phase ($k < T$), the convergence rate is bounded by the trajectory Hoffman constant, which may be quite large, although for ISTA, the active set tends to get smaller during this phase, so convergence may gain in speed.
    
    \item \textbf{Phase 2: Fast Active Set (Local Linear Convergence).} Once the correct support is identified (when $k \ge T$),  the algorithm enters the local phase and the iterates remain in the optimal subspace $\mathcal{S}_{\hat{\mathcal{A}}}$. The convergence rate is then governed by the local Hoffman constant, specifically, $\nu_{\text{local}} \propto 1/{H_{\hat{\mathcal{A}}}^2}$, which yields significantly faster linear convergence. This provides a theoretical justification for the fast tail often seen in convergence plots.
\end{itemize}

\paragraph{Quantitative Speed-Up and Restricted Smoothness.} To quantify the acceleration achieved during the local linear convergence phase, we compare the effective condition numbers $\kappa \propto L \cdot H^2$ in the global and restricted settings. 

The global smoothness constant scales as $L \sim d$. When combined with the ambient geometric penalty of the $\ell_1$ cross-polytope ($H_{\mathbb{R}^d} \sim \sqrt{d}$), the global condition number suffers a compounded dimensional penalty, scaling catastrophically as:
\begin{equation*}
    \kappa \propto L \cdot H_{\mathbb{R}^d}^2 \sim d \cdot (\sqrt{d})^2 = d^2.
\end{equation*}

However, once the iterates identify and restrict themselves to the optimal active subspace $\mathcal{S}_{\hat{\mathcal{A}}}$, both penalties are simultaneously neutralized:
\begin{itemize}
    \item \textbf{Geometric Improvement ($H_{\mathcal{K}}$):} Bypassing the dense faces of the ambient $\ell_1$ cross-polytope substantially reduces the Hoffman constant. The restricted constant scales as $H_{\hat{\mathcal{A}}} \sim \sqrt{n}$. This improves the squared geometric penalty from $d$ to $n$, yielding a purely geometric conditioning improvement proportional to $n/d$.
    
    \item \textbf{Spectral Improvement ($L_{\mathcal{K}}$):} The smoothness constant adapts strictly to the active submatrix. The global spectral penalty $L = \sigma_{\max}^2(A) \sim d$ reduces to the restricted smoothness $L_{\mathcal{K}} = \sigma_{\max}^2(A_{\hat{\mathcal{A}}})$. Because the active set size $s = |\hat{\mathcal{A}}|$ is much smaller than the ambient dimension $d$, the spectrum of the submatrix is highly concentrated. By similar singular value scaling arguments on the submatrix, $L_{\mathcal{K}} \sim n$, providing an additional $n/d$ spectral improvement.
\end{itemize}

Combining these two effects reveals the quantitative advantage of local active-set convergence. The effective local condition number scales strictly with the intrinsic problem size: 
\begin{equation*}
    \kappa_{\mathcal{K}} \propto L_{\mathcal{K}} \cdot H_{\hat{\mathcal{A}}}^2 \sim n \cdot (\sqrt{n})^2 = n^2.
\end{equation*}
Comparing this to the global scaling ($\kappa \sim d^2$), the effective condition number improves by a massive factor of $(n/d)^2$. This is significant in high-dimensional settings where $n \ll d$ and circumvents the ambient dimensional penalty.

\paragraph{Yet Another Local Geometric Structure.} The active-set perspective captures only part of the geometric structure explored by the optimization trajectory. A natural question is whether it is possible to describe this trajectory more precisely using continuous geometric constraints rather than discrete supports. For a broad class of decomposable $M$-estimators,~\citet[Lemmas 1 and 3]{Agarwal2012FastGlobal} show that iterates remain confined to structured sets around the optimal solution. These results suggest that the trajectory is governed by finer geometric sets than the active subspace alone, providing an opportunity for sharper, geometry-aware bounds. We now specialize their result to the restricted cone in the LASSO setting. 

\begin{lemma}[Iterated Cone Bound for LASSO]
\label{lem:iterated-cone-bound}
Let $\hat{\beta}$ be an optimal solution to the LASSO problem with support $\hat{\mathcal{A}}$. Let $\beta \in \mathbb{R}^d$ be any feasible vector satisfying the objective suboptimality bound $F(\beta) \le F(\hat{\beta}) + \delta$, where $\delta \ge 0$ is a chosen tolerance parameter. Assuming the regularization parameter $\eta$ satisfies the standard dual norm condition $\eta \ge 2 \|\nabla f(\beta^*)\|_\infty$ (where $\beta^*$ is the true underlying parameter), the error vector $\Delta := \beta - \hat{\beta}$ satisfies
\begin{equation*}
    \|\Delta_{\hat{\mathcal{A}}^c}\|_1 \le 3 \|\Delta_{\hat{\mathcal{A}}}\|_1 + 4 \|\hat{\beta}_{\hat{\mathcal{A}}^c}\|_1 + \frac{2\delta}{\eta}.
\end{equation*}
In the standard case where the target $\hat{\beta}$ is exactly supported on $\hat{\mathcal{A}}$ (i.e., $\hat{\beta}_{\hat{\mathcal{A}}^c} = 0$), this simplifies to the affine cone condition:
\begin{equation*}
    \|\Delta_{\hat{\mathcal{A}}^c}\|_1 \le 3 \|\Delta_{\hat{\mathcal{A}}}\|_1 + \frac{2\delta}{\eta}.
\end{equation*}
\end{lemma}

Note that the standard dual norm condition on $\eta$ is satisfied with high probability for the choice $\eta \gtrsim \sqrt{\frac{\log d}{n}}$ under standard sub-Gaussian noise assumptions \citep[e.g.,][]{Wainwright2019HighDimensional, Negahban2009UnifiedMEstimators}. 

\begin{figure}[ht]
    \centering
    \includegraphics[width=0.95\linewidth]{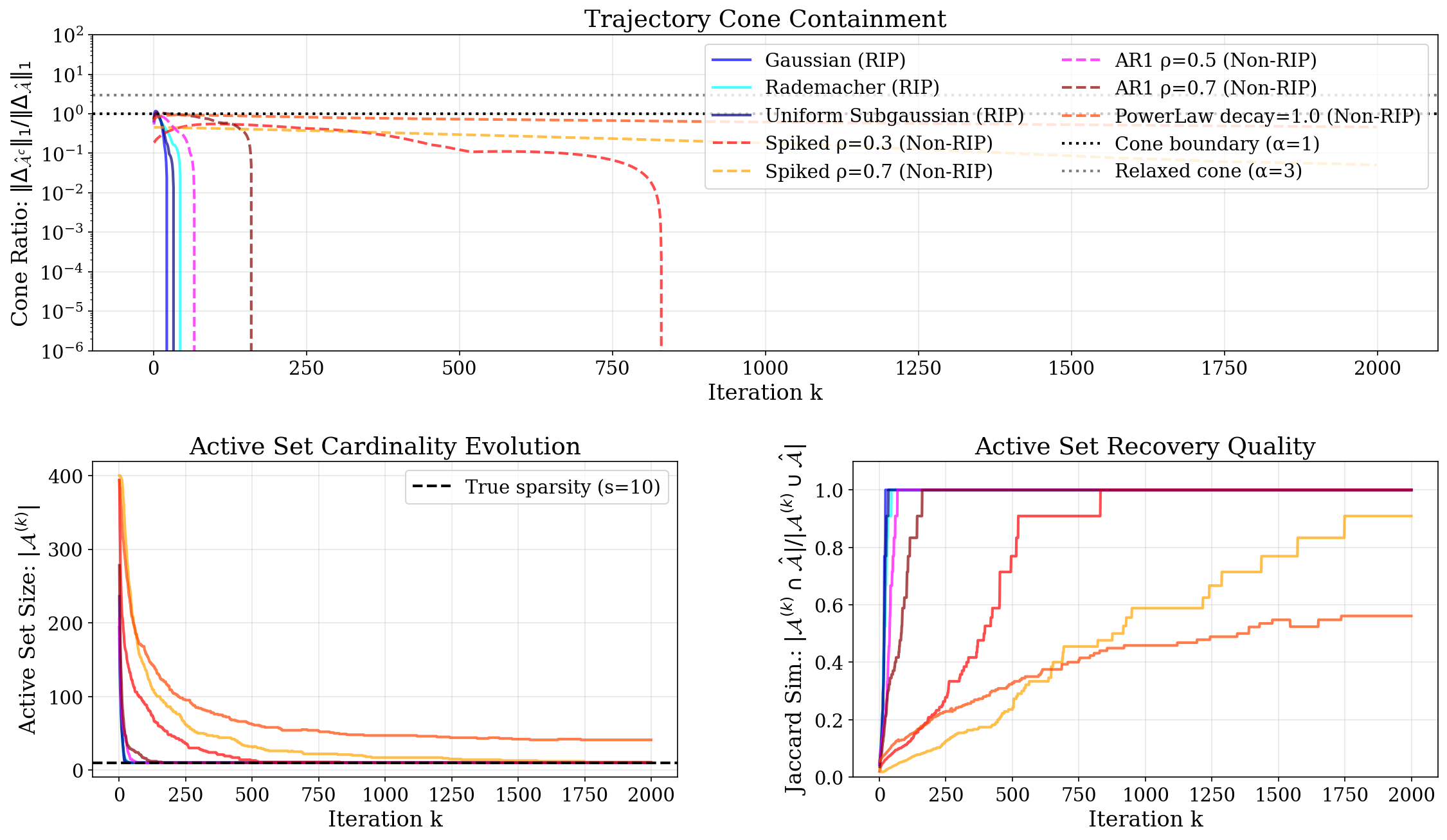}
    \caption[Trajectory containment of ISTA iterates for the normalized LASSO.]{\textbf{Trajectory containment of ISTA iterates for various random matrices ($n=200, d=400, s=10$) in the normalized LASSO problem.} Matrices which satisfy RIP are marked. \textbf{(Top)} The active cone ratio ($\|\Delta_{\hat{\mathcal{A}}^c} \|_1 / \|\Delta_{\hat{\mathcal{A}}} \|_1$) over iterations. We observe that trajectories tend to enter $\mathcal{C}_1(\hat{\mathcal{A}})$ relatively quickly. \textbf{(Bottom left)} Size of the current active set over iterations; tends to decrease. \textbf{(Bottom right)} Jaccard similarity between the current active set and the optimal active set over iterations. A similarity of 1 indicates the manifold has been identified.}
    \label{fig:trajectory-containment}
\end{figure}

The LASSO iterated cone bound lemma implies that the trajectory of any descent algorithm (such as ISTA) is eventually contained within a restricted cone around the optimal support. That is, the trajectory asymptotically enters the exact cone $\mathcal{C}_3(\hat{\mathcal{A}}) = \{ \Delta : \|\Delta_{\hat{\mathcal{A}}^c}\|_1 \le 3 \|\Delta_{\hat{\mathcal{A}}}\|_1 \}$. In fact, in~\cref{fig:trajectory-containment}, we empirically observe that ISTA trajectories remain within $\mathcal{C}_1(\hat{\mathcal{A}})$.

The empirical results in~\cref{fig:trajectory-containment} reveal three key geometric phenomena that corroborate our theoretical framework. First, the size of the active set $|\mathcal{A}_k|$ tends to decrease monotonically during the transient phase, suggesting that the local conditioning improves continuously throughout the trajectory rather than just asymptotically or after manifold identification. Second, manifold identification (where Jaccard similarity reaches $1.0$) coincides with the sharp drop in the cone ratio $\|\Delta_{\mathcal{A}^c}\|_1 / \|\Delta_{\mathcal{A}}\|_1$, validating the distinct transition from the search phase to the fast local phase. Finally, we observe a separation in convergence rates between designs satisfying the RIP and those that do not (e.g., spiked covariance with high correlation $\rho$)~\citep{Wainwright2019HighDimensional}; RIP designs enter the restricted cone $\mathcal{C}_1(\hat{\mathcal{A}})$ and converge soon after. The cone-restricted Hoffman constant may be uniformly bounded for such designs, which would explain the faster convergence rates observed in the compressed sensing literature~\citep[e.g.,][]{Xiao2013ProximalHomotopy}.

\paragraph{Normalized vs.\ Unnormalized Formulations.} We now raise a subtle but important distinction between the unnormalized objective
\begin{equation*}
    f(\beta) = \|A\beta - y\|_2^2 /2
\end{equation*}
and the normalized empirical risk
\begin{equation*}
    f(\beta) = \|A\beta - y\|_2^2 / (2n).
\end{equation*}
In the normalized formulation that is standard in statistical settings, the relevant matrix becomes the scaled Gram matrix $A^\top A / n$ and the key quantities are the normalized singular values.

Under mild assumptions (e.g., sub-Gaussian ensemble design)~\citep{Rudelson2010ExtremeSingularValues}, the largest singular value of the active submatrix $A_{\hat{\mathcal{A}}}$ scales as $\mathcal{O}(\sqrt{n})$, implying that the normalized constants
\begin{equation*}
    L_{\mathcal{K}} = \sigma_{\max}^2(A_{\hat{\mathcal{A}}}) / n
\end{equation*}
and
\begin{equation*}
    \nu_{\mathcal{K}} \propto 
1/(\sigma_{\min}^+(A_{\hat{\mathcal{A}}})^2 / n)
\end{equation*}
are both of order $\mathcal{O}(1)$. Consequently, the local condition number becomes dimension-free:
\begin{equation*}
    \kappa_{\mathcal{K}} = \mathcal{O}(1).
\end{equation*}

Under the normalized objective $f(\beta) = \|A\beta - y\|_2^2 / (2n)$, the relevant quantities are governed by the scaled Gram matrix $A^\top A / n$. In this setting, once the iterates are restricted to the active subspace $\mathcal{K} = \mathcal{S}_{\hat{\mathcal{A}}}$, both the smoothness and geometric constants become dimension-free:
\begin{equation*}
    L_{\mathcal{K}} = \frac{1}{n}\sigma_{\max}^2(A_{\hat{\mathcal{A}}})
\quad \text{and} \quad
H_{\mathcal{K}} \sim \left( \frac{1}{n}\sigma_{\min}^+(A_{\hat{\mathcal{A}}})^2 \right)^{-1},
\end{equation*}
which are $\mathcal{O}(1)$ under standard assumptions. Consequently, the effective local condition number satisfies
\begin{equation*}
    \kappa_{\mathcal{K}} = L_{\mathcal{K}} H_{\mathcal{K}}^2 = \mathcal{O}(1),
\end{equation*}
showing that the optimization landscape becomes scale-free after support identification. This suggests that the apparent ill-conditioning of high-dimensional problems is largely a global artifact and disappears along the optimization trajectory. In the case of high-dimensional sparse regression ($s \ll n \ll d$), the effective condition number transitions from being dependent on the full ambient dimension to depending only on the intrinsic problem size.\\

\subsection{Margin Classification: Support Vector Machines}

The framework extends naturally beyond LASSO to other problems with polyhedral structure. We first consider the standard primal formulation of the linear SVM with hinge loss and regularization parameter $\eta$:
\begin{equation*}
    \min_{\beta \in \mathbb{R}^d} \frac{\eta}{2} \|\beta\|_2^2 + \sum_{i=1}^n \max(0, 1 - y_i x_i^\top \beta).
\end{equation*}

While global linear convergence rates for SVMs have been established~\citep{Tseng2009BlockCoordinateSVM, Wang2014IterationComplexitySVM}, they typically depend on the smallest nonzero singular value of the full data matrix, which can be vanishingly small in high dimensions~\citep[Section 4.4]{Karimi2016PL}. Furthermore, this primal formulation embeds the data matrix inside the nonsmooth hinge loss.

In high-dimensional or kernelized settings, it is standard algorithmic practice to instead optimize the dual formulation via some coordinate method \citep{hsieh2008dual}. Conveniently, transitioning to the dual also perfectly aligns the SVM with our composite structure $F(x) = f(Ax) + g(x)$ and explicitly isolates the active set geometry. The dual SVM over the variables $\alpha \in \mathbb{R}^n$ is given by
\begin{equation*}
    \min_{\alpha \in \mathbb{R}^n} \underbrace{\frac{1}{2} \alpha^\top Q \alpha - \mathbf{1}^\top \alpha}_{f(\alpha)} + \underbrace{\delta_{\mathcal{X}}(\alpha)}_{g(\alpha)},
\end{equation*}
where $Q$ is the Gram matrix with entries $Q_{ij} = y_i y_j x_i^\top x_j$, and $\delta_{\mathcal{X}}$ is the indicator function for the polyhedral feasible set $\mathcal{X} = \{\alpha \mid y^\top \alpha = 0, \; 0 \le \alpha \le C\}$ (with $C = 1/\eta$). This convex formulation satisfies the restricted quadratic growth condition, which is sufficient for the restricted PL condition in the case of convex functions (see Appendix \ref{sec:QG}).

In this context, our restricted analysis focuses on the active set of the dual variables. The active set $\mathcal{A}(\alpha)$ corresponds to the non-zero dual variables ($\alpha_i > 0$), which represent exactly the \emph{support vectors} (the data points lying on or violating the margin).

The representer theorem~\citep{Schoelkopf2001Representer} shows that the minimizer of a regularized empirical risk over a reproducing kernel Hilbert space (RKHS) lies in the finite-dimensional span of the kernel functions evaluated at the training data. As a result, even when working in an infinite-dimensional feature space (e.g., with radial basis function kernels), the optimization problem reduces to estimating a finite set of coefficients $\alpha$. 

However, relying solely on the representer theorem or global quadratic growth is insufficient to explain the empirical fast convergence of SVM, as the global condition number of the resulting $n \times n$ Gram matrix often degrades rapidly with the sample size $n$ or the kernel width. Our restricted framework provides the missing link: once the optimization algorithm identifies the optimal active manifold corresponding to the set of true support vectors $\mathcal{S}^*$, the convergence rate is again governed strictly by the local Hoffman constant $H_{\mathcal{S}^*}$. This constant is determined by the geometric separation of the classes restricted to the subspace of support vectors. Consequently, the asymptotic linear rate depends on the spectral properties of the $s \times s$ submatrix of the kernel (where $s = |\mathcal{S}^*|$), effectively decoupling the convergence speed from both the infinite ambient dimension and the total sample size. To conclude, our framework provides a geometric explanation for an empirically observed phenomenon: kernelizing a problem does not necessarily degrade optimization speed when the solution remains sparse.

\subsection{Extensions to General Geometric Settings}

Beyond the polyhedral setting, the logic of restriction naturally extends to other structured problems with identifiable manifolds, such as nuclear-norm regularization where the local geometry of the low-rank manifold dictates the asymptotic rate once the structure is recovered. We refer to~\citet{Zhou2017ErrorBounds} for a discussion of local error bounds in such non-polyhedral cases.

Although we focused on proximal gradient methods and the PL inequality for brevity, our framework for restricted optimization conditions is general. Many proofs in the existing literature can be adapted to restricted versions almost immediately. For instance, the remaining conditions mentioned in~\citet[Theorem 2]{Karimi2016PL} can be naturally restricted to subsets $\mathcal{K}$. Similarly, results other problems and optimizers can be easily transposed into restricted variants relying on localized geometric constants; a non-exhaustive list includes: gradient descent, stochastic gradient descent (SGD), coordinate descent (randomized, proximal, greedy), and various boosting methods~\citep{Nutini2015Coordinate, Karimi2016PL, Meir2003Boosting}. 

While our analysis primarily focuses on Euclidean ($\ell_2$) trajectory-restricted regularity to characterize the linear convergence of proximal gradient methods, this framework generalizes immediately to arbitrary $\ell_p$ norms. By pairing the primal space measured in an $\ell_p$ norm with the dual space measured in the conjugate $\ell_q$ norm (where $1/p + 1/q = 1$), we can define an $\ell_p$-generalized gradient size:
\begin{equation*}
    \mathcal{D}_{g, p}(x, L_p) := -2 L_p \min_y \left[ \langle \nabla f(x), y - x \rangle + \frac{L_p}{2} \|y - x\|_p^2 + g(y) - g(x) \right].
\end{equation*}
Using this generalized gradient size, we define the restricted $\ell_p$-PL inequality for a subset $\mathcal{K}$ as
\begin{equation*}
    \frac{1}{2} \mathcal{D}_{g, p}(x, L_p) \ge \nu_{\mathcal{K}, p} (F(x) - F^*), \quad \forall\ x \in \mathcal{K}.
\end{equation*}
This generalization is applicable to different algorithmic solvers. For instance, coordinate descent algorithms update one feature at a time, moving strictly along canonical axes. Consequently, the steepest descent direction is measured in the $\ell_\infty$ dual norm, and the algorithm's geometric progress is governed by the $\ell_1$ primal geometry. Under our polyhedral framework, the corresponding $\ell_1$-PL constant~\citep{Nutini2015Coordinate} is tied to the restricted $\ell_1$-Hoffman constant $H_{\mathcal{K}}^{1, 1}$ of the active KKT system. Because bounding polyhedral systems in $\ell_1$ and $\ell_\infty$ norms suffers from severe dimensionality penalties compared to the $\ell_2$ norm~\citep[Exercises 7.5 and 7.14]{Wainwright2019HighDimensional}, this generalized $\ell_p$ framework provides a geometric explanation for why coordinate descent can stall on highly correlated high-dimensional designs, even when full-gradient methods maintain a stable $\ell_2$ rate.

Broadly, our results suggest a meta-theorem for optimization analysis: if it is possible to prove a geometric containment result in a well-conditioned section of the space, then tighter, geometry-aware convergence bounds follow. This perspective not only aids in designing new algorithms but also explains the remarkable effectiveness of existing methods where worst-case intuition is misaligned with empirical results. 

\section{Discussion}
\label{sec:discussion}

In this work, we have proposed a geometry-aware framework for analysing the linear convergence of proximal gradient methods. By shifting focus from global constants (which bound worst-case behavior) to restricted constants (which characterize the geometry of the active set), we bridge the significant gap between theoretical pessimism and empirical performance. Optimization algorithms do not explore the entire ambient space; they traverse specific, often favorable, low-dimensional or restricted paths. Our work formalizes this intuition, proving that if an algorithm can identify a well-conditioned subset, it inherits the fast convergence rate associated with that lower-dimensional subset, regardless of the ambient dimension.  

We conclude with several avenues for future research:
\begin{itemize}
    \item \textbf{Adaptive Manifold Identification:} A practical implication of our theory is the potential for geometry-informed adaptive algorithms. If an algorithm can detect when it has entered a favorable $\mathcal{K}$ (e.g., via stable active set patterns), it could switch to a more aggressive step-size (i.e., $1/L_{\mathcal{K}}$) or a second-order method tailored specifically to that lower-dimensional subspace.
    
    \item \textbf{Local Preconditioning:} The local Hoffman constant suggests that the ideal preconditioner changes as the active set evolves. Designing efficient preconditioners that estimate the local geometry without the prohibitive cost of full SVD computations remains an open challenge~\citep{Pena2024HoffmanHomogeneous}.
    
    \item \textbf{Probabilistic Trajectory Analysis:} While we have discussed deterministic containment, certain algorithms like SGD are unlikely to have such results. By combining high-dimensional concentration of measure with optimization dynamics, one might hope to prove that trajectories remain in favorable regions with high probability for random landscapes and algorithms.
    
    \item \textbf{Deep Nonconvexity:} Finally, there is a compelling analogy to be made in deep learning. Overparameterized neural networks are known to traverse low-dimensional manifolds during training~\citep{Mao2024LowDimManifold}. Characterizing the local geometry of these neural manifolds could provide new insights into why simple gradient methods succeed in such highly nonconvex, high-dimensional regimes.
\end{itemize}

\section*{Software Availability} Code to reproduce all figures and numerical experiments in this paper is provided in the following \href{https://github.com/farischaudhry/explicit-hoffman-constants}{GitHub repository}. The code to generate Figures~\ref{fig:hoffman-scaling} and~\ref{fig:trajectory-containment} is labeled experiment 4 and 3 respectively.

\section*{Acknowledgments}  A.M.~is supported by the EPSRC AI Hub on Mathematical Foundations of Intelligence: An ``Erlangen Programme'' for AI [EP/Y028872/1]. 
K.Y.~is supported by JSPS KAKENHI (24K15120, 24H00247, 26K02871).

\bibliographystyle{plainnat}
\bibliography{bibliography} 

@book{renegar2001mathematical,
  title={A mathematical view of interior-point methods in convex optimization},
  author={Renegar, James},
  year={2001},
  publisher={SIAM}
}

@article{todd1990dantzig,
  title={{A Dantzig-Wolfe-like variant of Karmarkar's interior-point linear programming algorithm}},
  author={Todd, Michael},
  journal={Operations Research},
  volume={38},
  number={6},
  pages={1006--1018},
  year={1990},
  publisher={INFORMS}
}

@InProceedings{Liao2024PMLR,
  title = 	 {Error bounds, {PL} condition, and quadratic growth for weakly convex functions, and linear convergences of proximal point methods},
  author =       {Liao, Feng-Yi and Ding, Lijun and Zheng, Yang},
  booktitle = 	 {Proceedings of the 6th Annual Learning for Dynamics and Control Conference},
  pages = 	 {993--1005},
  year = 	 {2024},
  editor = 	 {Abate, Alessandro and Cannon, Mark and Margellos, Kostas and Papachristodoulou, Antonis},
  volume = 	 {242},
  series = 	 {Proceedings of Machine Learning Research},
  month = 	 {15--17 Jul},
  publisher =    {PMLR}
}

@article{drusvyatskiy2021nonsmooth,
  title={{Nonsmooth optimization using Taylor-like models: error bounds, convergence, and termination criteria}},
  author={Drusvyatskiy, Dmitriy and Ioffe, Alexander and Lewis, Adrian},
  journal={Mathematical Programming},
  volume={185},
  number={1},
  pages={357--383},
  year={2021},
  publisher={Springer}
}

@article{stewart1989scaled,
  title={On scaled projections and pseudoinverses},
  author={Stewart, Gilbert},
  journal={Linear Algebra and its Applications},
  volume={112},
  pages={189--193},
  year={1989},
  publisher={Elsevier}
}

@article{davis2025active,
  title={Active manifolds, stratifications, and convergence to local minima in nonsmooth optimization},
  author={Davis, Damek and Drusvyatskiy, Dmitriy and Jiang, Liwei},
  journal={Foundations of Computational Mathematics},
  pages={1--83},
  year={2025},
  publisher={Springer}
}

@article{zhang2017restricted,
  title={The restricted strong convexity revisited: analysis of equivalence to error bound and quadratic growth},
  author={Zhang, Hui},
  journal={Optimization Letters},
  volume={11},
  number={4},
  pages={817--833},
  year={2017},
  publisher={Springer}
}

@article{anitescu2000degenerate,
  title={Degenerate nonlinear programming with a quadratic growth condition},
  author={Anitescu, Mihai},
  journal={SIAM Journal on Optimization},
  volume={10},
  number={4},
  pages={1116--1135},
  year={2000},
  publisher={SIAM}
}

@article{freund1999some,
  title={Some characterizations and properties of the ``distance to ill-posedness'' and the condition measure of a conic linear system},
  author={Freund, Robert and Vera, Jorge},
  journal={Mathematical Programming},
  volume={86},
  number={2},
  pages={225--260},
  year={1999},
  publisher={Springer}
}

@article{pena2000understanding,
  title={Understanding the geometry of infeasible perturbations of a conic linear system},
  author={Pe{\~n}a, Javier},
  journal={SIAM Journal on Optimization},
  volume={10},
  number={2},
  pages={534--550},
  year={2000},
  publisher={SIAM}
}

@inproceedings{Karimi2016PL,
  author  = {Karimi, Hamed and Nutini, Julie and Schmidt, Mark},
  title   = {{Linear convergence of gradient and proximal-gradient methods under the Polyak-\L{}ojasiewicz condition}},
  journal = {arXiv},
  year    = {2016},
  eprint  = {1608.04636},
  booktitle={Joint European
Conference on Machine Learning and Knowledge Discovery in Databases},
  url     = {https://doi.org/10.48550/arXiv.1608.04636}
}

@article{Tibshirani1996Lasso,
  title   = {Regression Shrinkage and Selection via the {L}asso},
  author  = {Tibshirani, Robert},
  journal = {Journal of the Royal Statistical Society: Series B (Methodological)},
  volume  = {58},
  number  = {1},
  pages   = {267--288},
  year    = {1996},
  url     = {https://doi.org/10.1111/j.2517-6161.1996.tb02080.x},
}

@article{Cortes1995SupportVector,
  author    = {Corinna Cortes and Vladimir Vapnik},
  title     = {Support-Vector Networks},
  journal   = {Machine Learning},
  year      = {1995},
  volume    = {20},
  number    = {3},
  pages     = {273--297},
  url       = {https://doi.org/10.1007/BF00994018},
  publisher = {Springer}
}

@inproceedings{Boser1992MarginClassifiers,
  author    = {Bernhard Boser and Isabelle Guyon and Vladimir Vapnik},
  title     = {A Training Algorithm for Optimal Margin Classifiers},
  booktitle = {Proceedings of the Fifth Annual Workshop on Computational Learning Theory (COLT '92)},
  year      = {1992},
  pages     = {144--152},
  url       = {https://doi.org/10.1145/130385.130401},
  publisher = {ACM},
  address   = {New York, NY, USA}
}

@article{Luo1993ErrorBounds,
  author    = {Luo, Zhi-Quan and Tseng, Paul},
  title     = {Error bounds and convergence analysis of feasible descent methods: a general approach},
  journal   = {Annals of Operations Research},
  volume    = {46},
  pages     = {157--178},
  year      = {1993},
  url       = {https://doi.org/10.1007/BF02096261}
}

@article{Drusvyatskiy2018ErrorBounds,
  title={Error bounds, quadratic growth, and linear convergence of proximal methods},
  author={Drusvyatskiy, Dmitriy and Lewis, Adrian},
  journal={Mathematics of Operations Research},
  volume={43},
  number={3},
  pages={919--948},
  year={2018},
  publisher={INFORMS},
  url={https://doi.org/10.1287/moor.2017.0889}
}

@article{Bolte2017ErrorBounds,
  author    = {Bolte, J{\'e}r{\^o}me and Nguyen, Pham and Peypouquet, Julien},
  title     = {From error bounds to the complexity of first-order descent methods for convex functions},
  journal   = {Mathematical Programming},
  volume    = {165},
  pages     = {471--507},
  year      = {2017},
  url       = {https://doi.org/10.1007/s10107-016-1091-6}
}

@article{Tseng2010GroupReg,
  author  = {Tseng, Paul},
  title   = {Approximation Accuracy, Gradient Methods, and Error Bound for Structured Convex Optimization},
  journal = {Mathematical Programming},
  volume  = {125},
  number  = {2},
  pages   = {263--295},
  year    = {2010},
  url     = {https://doi.org/10.1007/s10107-010-0394-2},
  publisher = {Springer}
}

@article{Zhang2013SparseGroupReg,
  author  = {Zhang, Haibin and Jiang, Jiaojiao and Luo, Zhi-Quan},
  title   = {On the Linear Convergence of a Proximal Gradient Method for a Class of Nonsmooth Convex Minimization Problems},
  journal = {Journal of the Operations Research Society of China},
  volume  = {1},
  number  = {2},
  pages   = {163--186},
  year    = {2013},
  url     = {https://doi.org/10.1007/s40305-013-0015-x},
  publisher = {Springer}
}

@inproceedings{Hou2013Nuclear,
  author    = {Hou, Ke and Zhou, Zirui and So, Anthony Man-Cho and Luo, Zhi-Quan},
  title     = {On the Linear Convergence of the Proximal Gradient Method for Trace Norm Regularization},
  booktitle = {Advances in Neural Information Processing Systems},
  volume    = {26},
  pages     = {710--718},
  year      = {2013},
  url       = {https://papers.nips.cc/paper_files/paper/2013/hash/41ae36ecb9b3eee609d05b90c14222fb-Abstract.html}
}

@article{Zhou2017ErrorBounds,
  title={A unified approach to error bounds for structured convex optimization problems},
  author={Zhou, Zhaosong and So, Anthony Man-Cho},
  journal={Mathematical Programming},
  volume={165},
  number={1},
  pages={689--728},
  year={2017},
  url={https://doi.org/10.1007/s10107-016-1100-9}
}

@book{Rockafellar1998Variational,
  author    = {Rockafellar, Tyrrell and Wets, Roger},
  title     = {Variational Analysis},
  series    = {Grundlehren der mathematischen Wissenschaften},
  volume    = {317},
  publisher = {Springer},
  address   = {Berlin, Heidelberg},
  year      = {1998},
  isbn      = {978-3-540-62772-2},
  url       = {https://doi.org/10.1007/978-3-642-02431-3}
}

@incollection{albano1999singularities,
  title={Singularities of semiconcave functions in {B}anach spaces},
  author={Albano, Paolo and Cannarsa, Piermarco},
  booktitle={Stochastic Analysis, Control, Optimization and Applications: A Volume in Honor of WH Fleming},
  pages={171--190},
  year={1999},
  publisher={Springer}
}

@article{peypouquet2010evi,
journal={Journal of Convex Analysis},
volume={17},
year={2010},
pages={1113--1163},
title={Evolution Equations for Maximal Monotone Operators: Asymptotic Analysis in Continuous and Discrete Time},
author={Peypouquet, Juan and Sorin, Sylvain}
}

@article{gupta2021path,
  title={Path length bounds for gradient descent and flow},
  author={Gupta, Chirag and Balakrishnan, Sivaraman and Ramdas, Aaditya},
  journal={Journal of Machine Learning Research},
  volume={22},
  number={68},
  pages={1--63},
  year={2021}
}

@book{brezis2011functional,
  title={Functional analysis, {S}obolev spaces and partial differential equations},
  author={Br{\'e}zis, Haim},
  year={2011},
  publisher={Springer}
}

@article{parikh2014proximal,
  title={Proximal algorithms},
  author={Parikh, Neal and Boyd, Stephen},
  journal={Foundations and Trends in optimization},
  volume={1},
  number={3},
  pages={127--239},
  year={2014},
  publisher={Emerald Publishing Limited}
}

@book{ambrosio2005gradient,
  title={Gradient flows: in metric spaces and in the space of probability measures},
  author={Ambrosio, Luigi and Gigli, Nicola and Savar{\'e}, Giuseppe},
  year={2005},
  publisher={Springer}
}

@article{Stella2017Forward,
  title={Forward--backward quasi-Newton methods for nonsmooth optimization problems},
  author={Stella, Lorenzo and Themelis, Andreas and Patrinos, Panagiotis},
  journal={Computational Optimization and Applications},
  volume={67},
  number={3},
  pages={443--487},
  year={2017},
  publisher={Springer},
  url={https://doi.org/10.1007/s10589-017-9912-y}
}

@article{Leventhal2008Randomized,
  author  = {Leventhal, Dennis and Lewis, Adrian},
  title   = {Randomized Methods for Linear Constraints: Convergence Rates and Conditioning},
  journal = {Mathematics of Operations Research},
  year    = {2010},
  volume={35},
  pages={641--654}
}

@article{Bolte2010characterizations,
  title={Characterizations of {\L}ojasiewicz inequalities: subgradient flows, talweg, convexity},
  author={Bolte, J{\'e}r{\^o}me and Daniilidis, Aris and Ley, Olivier and Mazet, Laurent},
  journal={Transactions of the American Mathematical Society},
  volume={362},
  number={6},
  pages={3319--3363},
  year={2010}
}

@article{Zhao2006ModelSelectionConsistency,
  author  = {Peng Zhao and Bin Yu},
  title   = {On Model Selection Consistency of {L}asso},
  journal = {Journal of Machine Learning Research},
  year    = {2006},
  volume  = {7},
  number  = {90},
  pages   = {2541--2563},
  url     = {http://jmlr.org/papers/v7/zhao06a.html}
}

@article{Hoffman1952ApproximateSolutions,
  title={On Approximate Solutions of Systems of Linear Inequalities},
  author={Hoffman, Alan},
  journal={Journal of Research of the National Bureau of Standards},
  volume={49},
  number={4},
  pages={263--265},
  year={1952},
  url={https://nvlpubs.nist.gov/nistpubs/jres/049/4/v49.n04.a05.pdf}
}

@inbook{Rudelson2010ExtremeSingularValues,
author = {Mark Rudelson and Roman Vershynin},
title = {Non-asymptotic Theory of Random Matrices: Extreme Singular Values},
booktitle = {Proceedings of the International Congress of Mathematicians 2010 (ICM 2010)},
chapter = {},
pages = {1576-1602},
year={2010},
publisher={World Scientific Publishing},
url = {https://doi.org/10.1142/9789814324359_0111},
}

@article{Nutini2019ActiveSet,
  author    = {Nutini, Julie and Schmidt, Mark and Hare, Warren},
  title     = {``Active-set complexity'' of proximal gradient: How long does it take to find the sparsity pattern?},
  journal   = {Optimization Letters},
  volume    = {13},
  pages     = {645--655},
  year      = {2019},
  url       = {https://doi.org/10.1007/s11590-018-1325-z}
}

@InProceedings{Sun2019ManifoldIdentification,
  title = 	 {Are we there yet? Manifold identification of gradient-related proximal methods},
  author =       {Sun, Yifan and Jeong, Halyun and Nutini, Julie and Schmidt, Mark},
  booktitle = 	 {Proceedings of the Twenty-Second International Conference on Artificial Intelligence and Statistics},
  pages = 	 {1110--1119},
  year = 	 {2019},
  volume = 	 {89},
  series = 	 {Proceedings of Machine Learning Research},
  publisher =    {PMLR},
  url = 	 {https://proceedings.mlr.press/v89/sun19a.html},
}

@article{HareLewis2004Active,
  author  = {Hare, Warren and Lewis, Adrian},
  title   = {Identifying Active Constraints via Partial Smoothness and Prox-Regularity},
  journal = {Journal of Convex Analysis},
  volume  = {11},
  number  = {2},
  pages   = {251--266},
  year    = {2004},
  publisher = {Heldermann Verlag}
}

@article{Lee2012ManifoldIdentification,
  author  = {Lee, Sangkyun and Wright, Stephen},
  title   = {Manifold Identification in Dual Averaging for Regularized Stochastic Online Learning},
  journal = {Journal of Machine Learning Research},
  volume  = {13},
  pages   = {1705--1744},
  year    = {2012},
  url     = {http://jmlr.org/papers/v13/lee12a.html}
}

@article{Burke1988ActiveConstraints,
  author  = {Burke, James and Mor{\'e}, Jorge},
  title   = {On the Identification of Active Constraints},
  journal = {SIAM Journal on Numerical Analysis},
  volume  = {25},
  number  = {5},
  pages   = {1197--1211},
  year    = {1988},
  url     = {https://doi.org/10.1137/0725068}
}

@article{Polyak1963Inequality,
title = {Gradient methods for the minimisation of functionals},
journal = {USSR Computational Mathematics and Mathematical Physics},
volume = {3},
number = {4},
pages = {864-878},
year = {1963},
issn = {0041-5553},
url = {https://doi.org/10.1016/0041-5553(63)90382-3},
author = {B.T. Polyak},
}

@article{Kurdyka1998,
  author  = {Kurdyka, Krzysztof},
  title   = {On gradients of functions definable in o-minimal structures},
  journal = {Annales de l'Institut Fourier},
  volume  = {48},
  number  = {3},
  pages   = {769--783},
  year    = {1998},
  doi     = {10.5802/aif.1629}
}

@article{Bickel2009LassoDantzigSelector,
  title={Simultaneous analysis of {L}asso and {D}antzig selector},
  author={Bickel, Peter and Ritov, Ya'acov and Tsybakov, Alexandre},
  journal={The Annals of Statistics},
  volume={37},
  number={4},
  pages={1705--1732},
  year={2009},
  url={https://doi.org/10.1214/08-AOS620},
}

@article{Candes2005DecodingLP,
  author  = {Cand{\`e}s, Emmanuel and Tao, Terence},
  title   = {Decoding by Linear Programming},
  journal = {IEEE Transactions on Information Theory},
  volume  = {51},
  number  = {12},
  pages   = {4203--4215},
  year    = {2005},
  url     = {https://doi.org/10.1109/TIT.2005.858979}
}

@inproceedings{Negahban2009UnifiedMEstimators,
 author = {Negahban, Sahand and Yu, Bin and Wainwright, Martin and Ravikumar, Pradeep},
 booktitle = {Advances in Neural Information Processing Systems},
 publisher = {Curran Associates, Inc.},
 title = {A unified framework for high-dimensional analysis of {$M$}-estimators with decomposable regularizers},
 url = {https://proceedings.neurips.cc/paper_files/paper/2009/file/dc58e3a306451c9d670adcd37004f48f-Paper.pdf},
 volume = {22},
 year = {2009}
}

@article{Lojasiewicz1963topological,
  title={A topological property of real analytic subsets},
  author={Łojasiewicz, Stanislaw},
  journal={Coll. du CNRS, Les {\'e}quations aux d{\'e}riv{\'e}es partielles},
  volume={117},
  number={87-89},
  pages={2},
  year={1963}
}

@article{bolte2014proximal,
  title={Proximal alternating linearized minimization for nonconvex and nonsmooth problems},
  author={Bolte, J{\'e}r{\^o}me and Sabach, Shoham and Teboulle, Marc},
  journal={Mathematical Programming},
  volume={146},
  number={1},
  pages={459--494},
  year={2014},
  publisher={Springer}
}

@article{Agarwal2012FastGlobal,
  title={Fast Global Convergence of Gradient Methods for High-Dimensional Statistical Recovery},
  author={Agarwal, Alekh and Negahban, Sahand and Wainwright, Martin},
  journal={The Annals of Statistics},
  volume={40},
  number={5},
  pages={2452--2482},
  year={2012},
  publisher={Institute of Mathematical Statistics},
  url={https://arxiv.org/abs/1104.4824}
}

@book{Wainwright2019HighDimensional,
  title={High-Dimensional Statistics: A Non-Asymptotic Viewpoint},
  author={Wainwright, Martin},
  series={Cambridge Series in Statistical and Probabilistic Mathematics},
  year={2019},
  publisher={Cambridge University Press},
  url={https://doi.org/10.1017/9781108627771}
}

@book{Vershynin2018HighDimensional, 
  series={Cambridge Series in Statistical and Probabilistic Mathematics}, 
  title={High-Dimensional Probability: An Introduction with Applications in Data Science}, 
  publisher={Cambridge University Press}, 
  author={Vershynin, Roman}, 
  year={2018}, 
  url={https://doi.org/10.1017/9781108231596}
}

@article{Bolte2007KLSubanalytic,
  author  = {Bolte, J{\'e}r{\^o}me and Daniilidis, Aris and Lewis, Adrian},
  title   = {The {\L}ojasiewicz Inequality for Nonsmooth Subanalytic Functions with Applications to Subgradient Dynamical Systems},
  journal = {SIAM Journal on Optimization},
  volume  = {17},
  number  = {4},
  pages   = {1205--1223},
  year    = {2007},
  url     = {https://doi.org/10.1137/050644641}
}

@article{Attouch2013DescentMethodsSemiAlgebraic,
  author  = {Attouch, H{\'e}dy and Bolte, J{\'e}r{\^o}me and Svaiter, Benar Fux},
  title   = {Convergence of descent methods for semi-algebraic and tame problems: proximal algorithms, forward--backward splitting, and regularized {G}auss--{S}eidel methods},
  journal = {Mathematical Programming},
  volume  = {137},
  number  = {1--2},
  pages   = {91--129},
  year    = {2013},
  url     = {https://doi.org/10.1007/s10107-011-0484-9}
}

@article{Robinson1981PolyhedralMultifunctions,
  title={Some continuity properties of polyhedral multifunctions},
  author={Robinson, Stephen},
  journal={Mathematical Programming at Oberwolfach},
  volume={14},
  pages={206--214},
  year={1981},
  publisher={Springer Berlin Heidelberg},
  url={https://doi.org/10.1007/BFb0120929}
}

@article{Pena2021Hoffman,
  author    = {Pe{\~n}a, Javier and Vera, Javier and Zuluaga, Luis},
  title     = {New Characterizations of {H}offman Constants for Systems of Linear Constraints},
  journal   = {Mathematical Programming},
  volume    = {187},
  pages     = {79--109},
  year      = {2021},
  doi       = {10.1007/s10107-020-01473-6},
  url       = {https://doi.org/10.1007/s10107-020-01473-6},
  publisher = {Springer},
}

@article{Wang2014IterationComplexitySVM,
  author  = {Po-Wei Wang and Chih-Jen Lin},
  title   = {Iteration Complexity of Feasible Descent Methods for Convex Optimization},
  journal = {Journal of Machine Learning Research},
  year    = {2014},
  volume  = {15},
  number  = {45},
  pages   = {1523--1548},
  url     = {http://jmlr.org/papers/v15/wang14a.html}
}

@article{Tseng2009BlockCoordinateSVM,
  author  = {Tseng, Paul and Yun, Sangwoon},
  title   = {Block-Coordinate Gradient Descent Method for Linearly Constrained Nonsmooth Separable Optimization},
  journal = {Journal of Optimization Theory and Applications},
  volume  = {140},
  number  = {3},
  pages   = {513--535},
  year    = {2009},
  url     = {https://doi.org/10.1007/s10957-008-9458-3}
}

@article{Xiao2013ProximalHomotopy,
  title={A Proximal-Gradient Homotopy Method for the Sparse Least-Squares Problem},
  author={Xiao, Lin and Zhang, Tong},
  journal={SIAM Journal on Optimization},
  volume={23},
  number={2},
  pages={1062--1091},
  year={2013},
  publisher={SIAM},
  url={https://doi.org/10.1137/12086999}
}

@article{Pena2024HoffmanHomogeneous,
  author  = {Peña, Javier},
  title   = {An easily computable upper bound on the {H}offman constant for homogeneous inequality systems},
  journal = {Computational Optimization and Applications},
  volume  = {87},
  number  = {2},
  pages   = {323--335},
  year    = {2024},
  url     = {https://doi.org/10.1007/s10589-023-00514-y}
}

@article{Mao2024LowDimManifold,
  author  = {Mao, Jialin and Griniasty, Itay and Teoh, Han Kheng and 
             Ramesh, Raghav and Yang, Rui and Transtrum, Mark and 
             Sethna, James and Chaudhari, Pratik},
  title   = {The training process of many deep networks explores the same low-dimensional manifold},
  journal = {Proceedings of the National Academy of Sciences},
  volume  = {121},
  number  = {12},
  pages   = {e2310002121},
  year    = {2024},
  url     = {https://doi.org/10.1073/pnas.2310002121}
}

@inproceedings{Nutini2015Coordinate,
  title={Coordinate Descent Converges Faster with the {G}auss-{S}outhwell Rule Than Random Selection},
  author={Nutini, Julie and Schmidt, Mark and Laradji, Issam and Friedlander, Michael and Koepke, Hoyt},
  booktitle={Proceedings of the 32nd International Conference on Machine Learning},
  series={JMLR},
  volume={37},
  year={2015},
  url={https://proceedings.mlr.press/v37/nutini15.pdf}
}

@incollection{Meir2003Boosting,
  title={An Introduction to Boosting and Leveraging},
  author={Meir, Ron and R{\"a}tsch, Gunnar},
  booktitle={Advanced Lectures on Machine Learning},
  editor={Mendelson, S. and Smola, A.J.},
  series={Lecture Notes in Computer Science},
  volume={2600},
  publisher={Springer},
  address={Berlin, Heidelberg},
  year={2003},
  url={https://doi.org/10.1007/3-540-36434-X_4}
}

@inproceedings{Schoelkopf2001Representer,
  author    = {Sch{\"o}lkopf, Bernhard and Herbrich, Ralf and Smola, Alex},
  title     = {A Generalized Representer Theorem},
  booktitle = {Computational Learning Theory (COLT 2001)},
  editor    = {Helmbold, David and Williamson, Bob},
  series    = {Lecture Notes in Computer Science},
  volume    = {2111},
  pages     = {416--426},
  publisher = {Springer},
  address   = {Berlin, Heidelberg},
  year      = {2001},
  url       = {https://doi.org/10.1007/3-540-44581-1_27}
}

@article{Damadi2022HardThresholding,
  author        = {Damadi, Saeed and Shen, Jinglai},
  title         = {Gradient Properties of Hard Thresholding Operator},
  journal       = {arXiv},
  year          = {2022},
  eprint        = {2209.08247},
  url           = {https://doi.org/10.48550/arXiv.2209.08247}
}

@article{Pang1997ErrorBounds,
  author    = {Jong-Shi Pang},
  title     = {Error bounds in mathematical programming},
  journal   = {Mathematical Programming},
  volume    = {79},
  number    = {1--3},
  pages     = {299--332},
  year      = {1997},
  publisher = {Springer},
  url       = {https://doi.org/10.1007/BF02614322}
}

@book{Beck2017FirstOrderMethods,
  author    = {Amir Beck},
  title     = {First-Order Methods in Optimization},
  year      = {2017},
  publisher = {Society for Industrial and Applied Mathematics},
  address   = {Philadelphia, PA},
  series    = {MOS-SIAM Series on Optimization},
  volume    = {25},
  url       = {https://doi.org/10.1137/1.9781611974997}
}

@book{Boyd2004ConvexOptimization,
  title     = {Convex Optimization},
  author    = {Boyd, Stephen and Vandenberghe, Lieven},
  year      = {2004},
  publisher = {Cambridge University Press},
  url       = {https://doi.org/10.1017/CBO9780511804441}
}

@inproceedings{hsieh2008dual,
  title={A Dual Coordinate Descent Method for Large-scale Linear SVM},
  author={Hsieh, Cho-Jui and Chang, Kai-Wei and Lin, Chih-Jen and Keerthi, S. Sathiya and Sundararajan, S.},
  booktitle={Proceedings of the 25th International Conference on Machine Learning (ICML)},
  pages={408--415},
  year={2008},
  address={Helsinki, Finland},
  url={https://doi.org/10.1145/1390156.1390208}
}

\clearpage
\appendix

\section{
Proof of Theorem \ref{thm:equivalence}
}
\label{app:equivalence}

In this appendix, we provide the proof for~\cref{thm:equivalence}, establishing the equivalence between the restricted EB and the restricted PL inequality as well as deriving the explicit equivalence constants.

\subsection{The Restricted Kurdyka--{\L}ojasiewicz Inequality}

Before proving~\cref{thm:equivalence},
we define the restricted Kurdyka--{\L}ojasiewicz (KL) inequality.
Recall the definition of the Fr\'echet (regular) subdifferential~\citep[Definition 8.3]{Rockafellar1998Variational}. For a real-valued function $F$, a vector $s$ lies in $\partial F(x)$ if
\begin{equation*}
    \liminf_{y \to x, y \neq x}
    \frac{F(y) - F(x) - \langle s, y - x \rangle}{\|y - x\|_2} \ge 0.
\end{equation*}
For the composite objective $F(x) = f(x) + g(x)$ with differentiable $f$ and convex $g$, the Fr\'echet subdifferential satisfies
\begin{equation*}
    \partial F(x) := \{ \nabla f(x) + \xi \mid \xi \in \partial g(x) \}.
\end{equation*}

\begin{definition}[Restricted KL Inequality with Exponent $1/2$]
    \label{def:restricted-kl-half}
    Let $\mathcal{K}\subseteq \mathbb{R}^{d}$.
    The function $F$ satisfies the \emph{$\mathcal{K}$-restricted K{\L} inequality with exponent $1/2$} if 
    there exists a constant $\tilde{\nu}_{\mathcal{K}} > 0$ such that
    \begin{equation}
    \label{eq:kl-half}
    \min_{s \in \partial F(x)} \|s\|_2^2 \;\ge\; 2\tilde{\nu}_{\mathcal{K}}\bigl(F(x)-F^*\bigr)
    \text{
for all $x \in \mathcal{K}$}.
\end{equation}
\end{definition}



\subsection{Proof of Theorem \ref{thm:equivalence}}

\begin{proof}

Let $x^+ := x - (1/L)\mathcal{G}_{1/L}(x)$ denote the proximal gradient update step, where $\mathcal{G}_{1/L}(x)$ is the proximal gradient mapping defined in~\eqref{eq:gradient-map} with step size $1/L$. 
    
We use the forward-backward envelope (FBE) of $F$ \cite[Definition 2.1]{Stella2017Forward}:
    \begin{equation}
    \label{eq:fbe}
        F_{1/L}(x) := f(x) + \langle \nabla f(x), x^+ - x \rangle + \frac{L}{2}\|x^+ - x\|_2^2 + g(x^+).
    \end{equation}
    Then the generalized gradient size $\mathcal{D}_g$ satisfies
    \begin{equation}
        \label{eq:Dg-identity}
        \frac{1}{2L} \mathcal{D}_g(x, L) = F(x) - F_{1/L}(x).
    \end{equation}
    
    \paragraph{Direction 1: Restricted Proximal EB $\implies$ Restricted Proximal PL.}
    Assume the $\mathcal{K}$-restricted proximal EB holds with constant $\mu_{\mathcal{K}}$:
    $\dist(x, \mathcal{X}^*) \le \mu_{\mathcal{K}} \|\mathcal{G}_{1/L}(x)\|_2$ for all $x \in \mathcal{K}$.
    
    We start with the upper bound on the suboptimality gap derived from the Lipschitz continuity of $\nabla f$. 
    Recall the definitions of the proximal gradient mapping in~\eqref{eq:gradient-map} and the FBE \eqref{eq:fbe}. Then we have that
    for any optimal solution $x^{*}=P_{\mathcal{X}^{*}}(x)$ with $P_{\mathcal{X}^{*}}(x)$ being the projection of $x$ onto $\mathcal{X}^{*}$,
    \begin{align*}
        F_{1/L}(x)-F^{*} 
        &= \min_{y\in\mathbb{R}^{d}}
        \left\{
        f(x)+\langle \nabla f(x),y-x\rangle + \frac{L}{2}\|y-x\|^{2}_{2}+g(y)
        \right\}-f(x^{*})-g(x^{*})\\
        &\le f(x)+\langle \nabla f(x),x^{*}-x\rangle + \frac{L}{2}\|x^{*}-x\|^{2}_{2}+g(x^{*})-f(x^{*})-g(x^{*})\\
        &=f(x)-f(x^{*})+\langle \nabla f(x),x^{*}-x\rangle + \frac{L}{2}\|x^{*}-x\|^{2}_{2}.
    \end{align*}
    The $L$-smoothness of $f$ gives
    \[
    f(x)-f(y) \le \langle \nabla f(y),x-y\rangle + \frac{L}{2}\|y-x\|^{2}_{2}, \quad \text{for all $x,y\in\mathbb{R}^{d}$}.
    \]
    Together with the Cauchy--Schwarz inequality, this implies 
    \begin{align*}
    f(x)-f(x^{*}) + \langle \nabla f(x),x^{*}-x\rangle 
    &\le 
    \langle \nabla f(x^{*})- \nabla f(x) , x^{*}-x \rangle +\frac{L}{2}\|x^{*}-x\|_{2}^{2}\\
    &\le \|\nabla f(x^{*})-\nabla f(x)\|_{2}\|x^{*}-x\|_{2}+\frac{L}{2}\|x^{*}-x\|_{2}^{2}\\
    &\le \frac{3L}{2}\|x^{*}-x\|^{2}_{2},
    \end{align*}
    which yields
    \begin{equation}
        \label{eq:fbe-suboptimality-bound}
        F_{1/L}(x) - F^* \le 2L \|x - x^*\|_2^2.
    \end{equation}
    Observe that
    $
    F(x)-F^{*} = (F(x) - F_{1/L}(x)) + (F_{1/L}(x) - F^*).
    $
    Using the identity~\eqref{eq:Dg-identity} for the first sum component and~\eqref{eq:fbe-suboptimality-bound} for the second, we obtain that for $x \in \mathcal{K}$,
    \begin{align}
        F(x) - F^* 
        &\le \frac{1}{2L} \mathcal{D}_g(x, L) + 2L \|x - x^*\|_2^2 \nonumber \\
        &\le \frac{1}{2L} \mathcal{D}_g(x, L) + 2L \mu_{\mathcal{K}}^2 \|\mathcal{G}_{1/L}(x)\|_2^2, \label{eq:dir1-inter}
    \end{align}
    where the second inequality applies the restricted proximal EB.
    
    We now show that the squared norm of the gradient mapping is bounded by $\mathcal{D}_g$. 
    The first-order optimality for the proximal subproblem defining $x^{+}$ yields
    \[
    0 \in \partial g(x^{+}) + L\left(x^{+}-x+\frac{1}{L}\nabla f(x)\right)
    = \partial g(x^{+}) + \nabla f(x) + L(x^{+}-x).
    \]
    Hence there exists a subgradient \(\xi^{+}\in \partial g(x^{+})\) such that
    $\xi^{+} + \nabla f(x) + L(x^{+}-x)=0.$
    Together with the convexity of $g$
    $g(z)\ge g(x^{+})+\langle \xi^{+},\, z-x^{+}\rangle$
    for any $z\in\mathbb{R}^{d}$, gives
    \[
    g(z)\ge g(x^{+})+\left\langle -\nabla f(x)-L(x^{+}-x),\, z-x^{+}\right\rangle.
    \]
    In particular, taking \(z=x\) leads to
    \[
    g(x)\ge g(x^{+})+\left\langle -\nabla f(x)-L(x^{+}-x),\, x-x^{+}\right\rangle.
    \]
    Rearranging, we have
    \[
    g(x^{+})-g(x)
    \le
    \langle \nabla f(x),\, x-x^{+}\rangle - L\|x-x^{+}\|^{2}.
    \]
    Therefore,
    \[
    g(x^{+})-g(x)+\langle \nabla f(x),\, x^{+}-x\rangle
    \le -L\|x-x^{+}\|^{2}.
    \]
    Substituting this into the definition of the FBE \eqref{eq:fbe}, we obtain
    \[
    F_{1/L}(x)-F(x)
    =
    g(x^{+})-g(x)
    +\langle \nabla f(x),\, x^{+}-x\rangle
    +\frac{L}{2}\|x^{+}-x\|^{2},
    \]
    from which we conclude that
    \begin{equation}
    \label{eq:bounding norm of gradmap}
    F(x)-F_{1/L}(x)\ge \frac{L}{2}\|x-x^{+}\|^{2}.
    \end{equation}

    Since $\|\mathcal{G}_{1/L}(x)\|_2 = L\|x - x^+\|_2$, the inequality \eqref{eq:bounding norm of gradmap} implies
    \begin{equation*}
        \|\mathcal{G}_{1/L}(x)\|_2^2 
        =L^2 \|x - x^+\|_2^2
        \le 2L (F(x) - F_{1/L}(x)) = \mathcal{D}_g(x, L).
    \end{equation*}
    Substituting this into~\eqref{eq:dir1-inter} gives
    \begin{equation*}
        F(x) - F^* \le \left( \frac{1}{2L} + 2L \mu_{\mathcal{K}}^2 \right) \mathcal{D}_g(x, L).
    \end{equation*}
    Rearranging this gives the $\mathcal{K}$-restricted proximal PL inequality:
    \begin{equation*}
        \frac{1}{2} \mathcal{D}_g(x, L) \ge \frac{L}{1 + 4L^2 \mu_{\mathcal{K}}^2}(F(x) - F^*).
    \end{equation*}
    That is, the proximal PL inequality holds with $\nu_{\mathcal{K}} = L/\{1 + 4L^2 \mu_{\mathcal{K}}^2\}$. 
    
    \paragraph{Direction 2: Restricted Proximal PL $\implies$ Restricted Proximal EB.} This direction relies on the connection to the restricted KL inequality.
    Assume the $\mathcal{K}$-restricted proximal PL inequality holds with constant $\nu_{\mathcal{K}}$.
    
    \textit{Step 1: Restricted Proximal PL implies Restricted KL.}
    Fix $x\in\mathcal{K}$ arbitrarily.
    We first bound the generalized gradient size $\mathcal{D}_g$ by the minimum $\ell_2$-norm subgradient. Let $\xi \in \partial g(x)$ be any subgradient. By definition, $g(y) - g(x) \ge \langle \xi, y - x \rangle.$ Substituting this into the definition of $\mathcal{D}_g$ gives
    \begin{equation*}
        \mathcal{D}_g(x, L) \le -2L \min_{y\in\mathbb{R}^{d}} \left[ \langle \nabla f(x) + \xi, y - x \rangle + \frac{L}{2} \|y - x\|_2^2 \right].
    \end{equation*}
    The minimization on the RHS yields exactly $-(1/2L) \|\nabla f(x) + \xi\|_2^2$. Thus, we have 
    \[\mathcal{D}_g(x, L) \le \|\nabla f(x) + \xi\|_2^2
    \quad\text{for any $\xi \in \partial g(x)$}.\]
    Since the Fr\'{e}chet subdifferential of $F$ satisfies $\partial F(x)=\{\nabla f(x)+\xi \mid \xi \in\partial g(x)\},$
    this implies 
    \[
    \mathcal{D}_g(x, L) \le \min_{s\in \partial F(x)} \|s\|_{2}^{2},
    \]
    which, together with the restricted proximal PL inequality, $\frac{1}{2}\mathcal{D}_g(x, L) \ge \nu_{\mathcal{K}}(F(x) - F^*)$,
    yields the $\mathcal{K}$-restricted KL inequality with $\tilde{\nu}_{\mathcal{K}} = \nu_{\mathcal{K}}$:
    \[
    \min_{s\in \partial F(x)} \|s\|_{2}^{2} \ge 2\nu_{\mathcal{K}}(F(x) - F^*).
    \]

    \vspace{3mm}
    \textit{Step 2: Taking a subgradient flow staying in $\mathcal{K}$.}
We next take a subgradient flow starting from $x\in\mathcal{K}$ and staying in $\mathcal{K}$.
To this end, we recall the notion of semi-convexity. A function $F$ is said to be \emph{semi-convex} if there exists $\rho > 0$ such that $F(x) + \rho \|x\|_2^2$ is convex.
    The composite objective $F$ is semi-convex from the smoothness of $f$ and the convexity of $g$.
    Thus, for any $x\in\mathcal{K}\cap \mathop{\mathrm{dom}}(F)$, 
    there exists a unique subgradient flow $\chi_{x}:[0,\infty)\to\mathop{\mathrm{dom}}(F)$  \citep[Theorem 13]{Bolte2010characterizations}, that is, the curve satisfying the differential inclusion
    \begin{equation*}
        \dot{\chi}_{x}(t) \in -\partial F(\chi_{x}(t))
        \quad
        \text{for a.e.}~t\in[0,\infty)
        \text{ and }
        \chi_{x}(0) = x.
    \end{equation*}
    This unique subgradient flow satisfies
    \begin{equation}
    \label{eq:energy decrease of subgradient flow}
        \frac{d}{dt}F(\chi_{x}(t)) = -\|\dot{\chi}_{x}(t)\|_{2}^{2}
        \quad
        \text{for a.e.~$t\in[0,\infty)$},
    \end{equation}
    where $F(\chi_{x}(t))$ is non-increasing and Lipschitz-continuous with respect to $t$ on $[\eta,\infty]$ for any $\eta>0$.

The following lemma ensures that 
under Assumptions~\ref{assum:stability of proximal gradient update} and \ref{assum:Kclosed},
the unique subgradient flow $\chi_{x}(\cdot)$ starting from a point in $\mathcal{K}$ remains in $\mathcal{K}$.
    
    \begin{lemma}[Flow Invariance of $\mathcal K$]
\label{lem:flow-invariance}
Under Assumptions~\ref{assum:stability of proximal gradient update} and \ref{assum:Kclosed},
the unique subgradient flow of $F$ starting from a point in $\mathcal K$ remains in $\mathcal K$.
\end{lemma}

The proof for \Cref{lem:flow-invariance} is given after we conclude the main proof of \Cref{thm:equivalence}; it uses an approximation of the subgradient flow by using the proximal gradient sequence, which is slightly different from the standard proximal-point approximation~\citep{ambrosio2005gradient,peypouquet2010evi}.

\vspace{3mm}
\textit{Step 3: Constructing a subdifferential error bound using the subgradient flow.}
We now show
    \begin{equation}
    \label{eq:bounding min norm of subgrads}
    \min_{s\in\partial F(x)} \|s\|_{2} \ge  \nu_{\mathcal{K}}\dist(x,\mathcal{X}^{*}) \quad\text{for any $x\in\mathcal{K}\cap \dom(F)$}.
    \end{equation}
Define $r(t):=\sqrt{F(\chi_{x}(t))-F^{*}}$ for $t\in [0,\infty)$.
    By the definition of subgradient flow and the restricted K{\L} inequality, we have
    \begin{align*}
        \dot{r}(t) = \frac{\dot{F}(\chi_{x}(t))}{2\sqrt{F(\chi_{x}(t))-F^{*}}}
        = -\frac{\|\dot{\chi_{x}}(t)\|^{2}_{2}}{2\sqrt{F(\chi_{x}(t))-F^{*}}}
        \le -\sqrt{\nu_{\mathcal{K}}/2}\|\dot{\chi}_{x}(t)\|_{2}.
    \end{align*}
    Integrating this expression from 0 to any $T>0$ gives
    \begin{align}
    \label{eq: length inequality}
        r(T)-r(0) = \int_{0}^{T}\dot{r}(t)dt 
        \le 
        - \sqrt{\frac{\nu_{\mathcal{K}}}{2}}\int_{0}^{T}\|\dot{\chi}_{x}(t)\|dt
        \le - \sqrt{\frac{\nu_{\mathcal{K}}}{2}}\|\chi_{x}(T)-x\|.
    \end{align}

We now prove that as $r(T)\to 0$,
$\chi_{x}(T)$ converges to some $x_{\infty}$ as $T\to\infty$ and
$x_{\infty}\in \mathcal{X}^{*}$ as follows.
    Observe that
    \[
    \dot{r}(t) = 
    -\frac{\|\dot{\chi}_{x}(t)\|_{2}^{2}}{2\sqrt{F(\chi_{x}(t))-F^{*}}} \le - \frac{\nu_{\mathcal{K}}}{2}r(t), \quad t\in[\eta,T],
    \]
    where $\eta\in (0,T)$ is taken arbitrarily and we use the restricted KL inequality.
    Together with the Gr\"{o}nwall inequality, this gives
    \[
    r(t)\le r(\eta) \exp(-\nu_{\mathcal{K}}(t-\eta)/2)
    \] and  $r(T)\to0$. 
    Next, we show $\chi_{x}(T)$ converges to some $x_{\infty}$.
    For any $s>t\ge 0$, we have
\[
\|\chi_x(s)-\chi_x(t)\|_2
\le
\int_t^s \|\dot{\chi}_x(\tau)\|_2\,d\tau.
\]
On the other hand, integrating $\dot r(\tau)\le -\sqrt{\nu_{\mathcal K}/2}\,\|\dot{\chi}_x(\tau)\|_2$ from $t$ to $s$ yields
\[
\int_t^s \|\dot{\chi}_x(\tau)\|_2\,d\tau
\le
\sqrt{\frac{2}{\nu_{\mathcal K}}}\,(r(t)-r(s))
\le
\sqrt{\frac{2}{\nu_{\mathcal K}}}\,r(t),
\]
implying 
\[
\|\chi_x(s)-\chi_x(t)\|_2
\le
\sqrt{\frac{2}{\nu_{\mathcal K}}}\,r(t).
\]
Since $r(t)\to 0$, this shows that the flow $\chi_x(t)$ is Cauchy, hence it converges to some $x_\infty\in\mathbb R^d$.
Finally, by lower semicontinuity of $F$, we obtain
\[
F(x_\infty)\le \liminf_{t\to\infty}F(\chi_x(t))=F^*,
\]
which implies $x_\infty\in\mathcal X^*$.

    From $r(T)\to 0$ and $\lim_{T\to\infty}\chi_{x}(T)\in \mathcal{X}^{*}$, \eqref{eq: length inequality} yields
    \[
    r(0)=\sqrt{F(x)-F^{*}} \ge \sqrt{\frac{\nu_{\mathcal{K}}}{2}}\dist(x,\mathcal{X}^{*}),
    \]
    which, together with the restricted KL inequality, leads to the inequality \eqref{eq:bounding min norm of subgrads} as desired.

\vspace{3mm}
    \textit{Step 4: From the subdifferential error bound to the restricted EB.} Now, the first-order optimality for the proximal subproblem defining $x^{+}$ yields
    \[
    0 \in \partial g(x^{+}) + L\left(x^{+}-x+\frac{1}{L}\nabla f(x)\right)
    = \partial g(x^{+}) + \nabla f(x) + L(x^{+}-x).
    \]
    Hence there exists a subgradient \(\xi^{+}\in \partial g(x^{+})\) such that
    $\xi^{+} + \nabla f(x) + L(x^{+}-x)=0.$
    Using this, we have
    \begin{equation}
    \label{eq: bounding min subgrad above}
    \min_{s\in\partial F(x^{+})} \|s\|_{2} \le \|\nabla f(x^{+})+\xi^{+}\|
    \le \|\nabla f(x)-\nabla f(x^{+})\|+L\|x-x^{+}\|
    \le 2L \|x-x^{+}\|.
    \end{equation}
    Together with the assumption that $x^{+}$ remains in $\mathcal{K}$,
    combining this with \eqref{eq:bounding min norm of subgrads} for $x^{+}$ gives
    \[
\nu_{\mathcal{K}} \cdot \dist(x,\mathcal{X}^{*}) \le 2L\|x-x^{+}\|.
    \]
    By the triangle inequality, $\dist(x, \mathcal{X}^*) \le \|x-x^{+}\|_{2} + \dist(x^{+},\mathcal{X}^{*}).$
    Thus, we have
    \[
    \dist(x, \mathcal{X}^*) 
    \le \left(1 + \frac{2L}{\nu_{\mathcal{K}}}\right) 
    \|x - x^+\|_2.
    \]
    Recall $\|x - x^+\|_2 = (1/L) \|\mathcal{G}_{1/L}(x)\|_2$. Thus, the $\mathcal{K}$-restricted proximal EB holds with constant
        \begin{equation*}
            \mu_{\mathcal{K}} = \frac{1}{L} + \frac{2}{\nu_{\mathcal{K}}},
        \end{equation*}
    which proves Theorem \ref{thm:equivalence}.
\end{proof}

\begin{proof}[Proof of Lemma \ref{lem:flow-invariance}]
    We begin with constructing a curve sequence staying in $\mathcal{K}$ and converging to the subgradient flow $\chi_{x}(t)$ in $[0,T]$, where $T$ is arbitrarily prefixed.
    To do this, we use the proximal gradient sequence and the interpoland from this sequence.
    Let $x_{0}=x\in\mathcal{K}$.
    For $\lambda\in (0,1/(2L))$, let
    \[
    x_{k+1}:=\prox_{\lambda g}\bigl(x_k-\lambda\nabla f(x_k)\bigr), \quad k=0,1,\ldots, K-1 \quad \text{with}\;K=\lfloor T/\lambda\rfloor. \]
    By Assumption~\ref{assum:stability of proximal gradient update}, this proximal gradient sequence stays in $\mathcal{K}$.
    By the optimality condition of the proximal step,
    we have
    \[
0\in \partial g(x_{k+1})+\frac{1}{\lambda}\bigl(x_{k+1}-x_k+\lambda\nabla f(x_k)\bigr).\]
Rearranging yields
\[
\frac{x_{k+1}-x_k}{\lambda}
+\nabla f(x_k)-\nabla f(x_{k+1})
\in -\partial F(x_{k+1}).
\]
Here, we consider a piecewise affine interpolant
\begin{align*}
u_{\lambda}(t)&:=x_{k}+(t-k\lambda)\frac{x_{k+1}-x_{k}}{\lambda}, \quad t\in [k\lambda, (k+1)\lambda).
\end{align*}
Observe that
\[
\dot{u}_{\lambda}(t)=\frac{x_{k+1}-x_{k}}{\lambda},\quad t\in (k\lambda,(k+1)\lambda).
\]
Then the squared path length satisfies
\[
\int_{0}^{T} \|\dot{u}_{\lambda}(t)\|^{2}_{2}dt
= \sum_{k=0}^{K-1} \frac{\|x_{k+1}-x_{k}\|^{2}_{2}}{\lambda}.
\]
By the proximal gradient descent lemma \citep[Theorem 5]{Karimi2016PL},
we have that for $\lambda\le1/L$,
\[
F(x_{k+1}) \le F(x_k) - \frac{\lambda}{2} \mathcal{D}_g(x_k, 1/\lambda).
\]
Also, we have the following lower bound on $\mathcal{D}_{g}(x_{k},1/\lambda)$:
\[
\frac{\lambda}{2} \mathcal{D}_g(x_k, 1/\lambda)\ge \frac{1-\lambda L}{2\lambda}\|x_{k+1}-x_{k}\|_{2}^{2}.
\]
Together, we have
\[
        F(x_{k+1}) \le F(x_k) - \frac{1-\lambda L}{2\lambda}\|x_{k+1}-x_{k}\|_{2}^{2}.
\]
Since $\lambda\le 1/(2L)$, we have $(1-\lambda L)/(2\lambda)\ge 1/(4\lambda)$, so summing over
$k=0,\dots,K-1$ yields
\begin{equation}
\label{eq: path length formula}
\sum_{k=0}^{K-1}
\frac{\|x_{k+1}-x_k\|_2^2}{\lambda}
\le 4\bigl(F(x)-F^{*}\bigr).
\end{equation}
The LHS is the squared path length and thus we obtain $\int_0^T \|\dot u_\lambda(t)\|_2^2\,dt \le 4\bigl(F(x)-F^*\bigr).$
Also, we have
\[
\int_{0}^{T} \|u_{\lambda}(t)\|^{2} dt \le 2\lambda\sum_{k=0}^{K-1}\|x_{k}\|_{2}^{2} + 2\lambda^{2}\sum_{k=0}^{K-1}
\frac{\|x_{k+1}-x_{k}\|_{2}^{2}}{\lambda} \le 2\lambda\sum_{k=0}^{K-1}\|x_{k}\|_{2}^{2} + 8\lambda^{2}(F(x)-F^{*}).
\]
Here, from the nonexpansiveness of the proximal operator \citep[e.g.,][]{parikh2014proximal}, we have that for arbitrary $x^{*}\in\mathcal{X}^{*}$,
\[
\|x_{k+1}-x^{*}\|_{2}
\le \|(x_{k}-x^{*})-\lambda(\nabla f(x_{k})-\nabla f(x^{*}))\|_{2}
\le (1+\lambda L)\|x_{k}-x^{*}\|, \quad k=0,\ldots,K-1,
\]
where the last inequality follows from the $L$-Lipschitz continuity of $\nabla f$.
Then the discrete Gr\"{o}nwall inequality gives $\|x_{k}\|_{2} \le \|x^{*}\|_{2} + \mathrm{e}^{LT}\|x-x^{*}\|_{2}.$ Hence, the Sobolev norm of $u_{\lambda}(\cdot)$ is uniformly bounded:
\[
\sup_{\lambda\in(0,1/2L)}
\left\{\int_0^T \| u_\lambda(t)\|_2^2\,dt 
+
\int_0^T \|\dot u_\lambda(t)\|_2^2\,dt \right\} \le C_{x,T}<\infty.
\]
Thus we can take a subsequence $u_{\lambda}$ with 
$ u_{\lambda} \to u$ and $\dot{u}_{\lambda}\to\dot{u}$ in the weak sense.
Here by the Rellich--Kondrachov theorem \citep[Theorem 9.16]{brezis2011functional}, we also have $\sup_{t\in [0,T]}|u_{\lambda}(t)- u(t)|\to 0$. 

We will show that $u(t)$ is a subgradient flow and $u(t)$ is $\chi_{x}(t)$ by its uniqueness.
Since $F$ is semi-convex, the differential inclusion $\dot u(t)\in -\partial F(u(t))$
for a.e.~$t\in(0,T)$ is equivalent to the variational inequality
\[
F(z)\ge F(u(t))+\langle -\dot u(t), z-u(t)\rangle
-\frac{\rho}{2}\|z-u(t)\|_{2}^{2}
\quad\text{for all }z\in\mathbb R^d
\;\text{and for a.e.~$t\in(0,T)$};
\]
see \citet{albano1999singularities,ambrosio2005gradient} for details.
Therefore, it suffices to prove that for arbitrary nonnegative $\phi(t)\in L^{2}(0,T)$,
\begin{align}
\label{eq: target EVI}
\int_{0}^{T} \phi(t) F(z) dt
\ge \int_{0}^{T} \phi(t)F(u(t))dt
+ \int_{0}^{T}\phi(t)\langle -\dot u(t), z-u(t)\rangle dt
-\frac{\rho}{2}\int_{0}^{T}
\phi(t)\|z-u(t)\|_{2}^{2}dt.
\end{align}
Fix nonnegative $\phi(t)\in L^{2}(0,T)$ arbitrarily.
By the optimality condition of the proximal step, we have
\[
w_{\lambda}(t):=\dot{u}_{\lambda}(t)
+\nabla f(x_k)-\nabla f(x_{k+1})
\in -\partial F(x_{k+1}).
\]
Setting $y_{\lambda}(t):=x_{k+1},\, t\in [k\lambda,(k+1)\lambda)$, this is equivalent to
\begin{align}
\label{eq: variational inequality in y and w}
F(z) &\ge F(y_{\lambda}(t)) + \langle -w_{\lambda}(t), z-y_{\lambda}(t)\rangle
-\frac{\rho}{2}\|z-y_{\lambda}(t)\|_{2}^{2}
\quad
\text{for all }z\in\mathbb R^d
\;
\text{and for a.e.~$t\in(0,T)$}.
\end{align}
Observe that
\begin{align}
\sup_{t\in(0,T)}\|w_{\lambda}(t)-\dot{u}_{\lambda}(t)\|^{2}_{2} 
&\le \max_{k=0,\ldots,K-1}\|\nabla f(x_{k})-\nabla f(x_{k+1})\|_{2}^{2}\nonumber\\
&\le \max_{k=0,\ldots,K-1}L^{2}\|x_{k+1}-x_{k}\|^{2}_{2} \label{ineq1:lemproof} \\
&\le L^{2}\lambda \sum_{k=0}^{K -1}
\frac{\|x_{k+1}-x_k\|_2^2}{\lambda} \nonumber\\
&\le 4L^{2}\lambda\bigl(F(x)-F^{*}\bigr) \to 0 \quad\text{as $\lambda\to 0$}, \label{ineq2:lemproof}
\end{align}
where the second inequality \eqref{ineq1:lemproof} follows from the smoothness of $f$, and the last inequality \eqref{ineq2:lemproof} follows from \eqref{eq: path length formula}.
Together with the weak convergence of $\dot{u}_{\lambda}$, this implies $w_{\lambda}\to \dot{u}$ in the weak sense
up to a subsequence.
Further, from \eqref{eq: path length formula}, we have
\[
\sup_{t\in(0,T)}\|u_{\lambda}(t)-y_{\lambda}(t)\|_{2}^{2}\le \max_{k=0,\ldots,K-1}\|x_{k+1}-x_{k}\|^{2}_{2} \le 4\lambda (F(x)-F^{*}),
\]
implying $y_{\lambda}\to u$ in $C(0,T)$ up to a subsequence.
By the lower semicontinuity of $F$
and the Fatou lemma,
taking a limit with respect to $\lambda$ gives
\[
\liminf_{\lambda\to 0} \int_{0}^{T}\phi(t) F(y_{\lambda}(t))dt \ge \int_{0}^{T} \phi(t)F(u(t))dt.
\]
Also, by the weak convergence of $w_{\lambda}$ and by the strong convergence of $y_{\lambda}$, we have
\begin{align*}
\int_{0}^{T}\phi(t)\langle 
-w_{\lambda}(t),z-y_{\lambda}(t)
\rangle
&=\int_{0}^{T}\phi(t)\langle 
-w_{\lambda}(t)+\dot{u}(t)-\dot{u}(t),z-u(t)+u(t)-y_{\lambda}(t)
\rangle dt
\\
&\to \int_{0}^{T}\phi(t)\langle -\dot{u}(t),z-u(t)\rangle dt \quad \text{as $\lambda\to 0$}.
\end{align*}
Therefore we obtain \eqref{eq: target EVI}
and the differential inclusion for $u$.
By the construction, we have $u(0) = x_{0} =x.$ Hence, this $u(t)$ is a subgradient flow. By the uniqueness of the subgradient flow, $u(t)$ is thus $\chi_{x}(t)$.
Since $y_\lambda(t)\in\mathcal K$ for all $t\in[0,T]$, and since $y_\lambda\to u=\chi_x$
uniformly on $[0,T]$, the closedness of $\mathcal K$ (Assumption~\ref{assum:Kclosed}) implies
\[
\chi_x(t)\in \mathcal K
\qquad \text{for } t\in[0,T],
\]
which completes the proof.
\end{proof}

\section{
Connection to the Restricted Quadratic Growth Condition
}
\label{sec:QG}

In this appendix, 
we show that the restricted quadratic growth condition defined below implies the restricted PL condition. 

\begin{definition}[Restricted Quadratic Growth (QG)]
\label{def:restricted-qg}
The objective function $F$ satisfies the \emph{$\mathcal{K}$-restricted quadratic growth condition with constant $\gamma_{\mathcal{K}} > 0$} if
$$
    F(x) - F^* \ge \frac{\gamma_{\mathcal{K}}}{2} \dist(x, \mathcal{X}^*)^2 \quad \text{for all $x \in \mathcal{K}\cap \dom (F)$},
$$
where $\mathcal{X}^*$ is the nonempty optimal set.
\end{definition}

In the case of a convex objective function, the restricted QG condition is a sufficient condition for the restricted PL inequality.

\begin{lemma}[Restricted QG implies Restricted PL]
\label{lem:qg-to-pl}
Assume that $L\ge \gamma_{\mathcal{K}}/2$.
If a convex function $F$ satisfies the $\mathcal{K}$-restricted QG condition with constant $\gamma_{\mathcal{K}}$, then it satisfies the $\mathcal{K}$-restricted proximal PL inequality with constant $\nu_{\mathcal{K}} = \gamma_{\mathcal{K}}/2$.
\end{lemma}
\begin{proof}
    Let $x \in \mathcal{K}\cap \dom (F)$ and let $x^* = P_{\mathcal{X}^*}(x)$ be the projection onto the optimal set. Recall the definition of the generalized gradient size for a step size $\lambda > 0$:
    \begin{equation*}
        \mathcal{D}_g(x, \lambda) = -2\lambda \min_{y\in\mathbb{R}^{d}} \Big{[} \underbrace{\langle \nabla f(x), y - x \rangle + \frac{\lambda}{2}\|y - x\|^2_2 + g(y) - g(x)}_{=:Q_\lambda(y)} \Big{]}.
    \end{equation*}
    Let $Q_\lambda(y)$ denote the term inside the minimization. Since $x^*$ is a feasible point, the minimum over all $y$ is bounded above by the value at $x^*$, i.e., 
    \begin{align}
    \label{eq: bounding Q}
    \min_y Q_\lambda(y) \le Q_\lambda(x^*).
    \end{align}
    
Evaluating $Q_\lambda(y)$ at $x^*$ gives
   $$
        Q_\lambda(x^*) = \langle \nabla f(x), x^* - x \rangle + \frac{\lambda}{2}\|x^* - x\|^2_2 + g(x^*) - g(x).
   $$
    By the convexity of $f$, we have the inequality $\langle \nabla f(x), x^* - x \rangle \le f(x^*) - f(x)$. Substituting this into the expression above yields
    \begin{align}
        Q_\lambda(x^*) &\le f(x^*) - f(x) + \frac{\lambda}{2}\|x^* - x\|^2_2 + g(x^*) - g(x) \label{ineq:qg}\\
        &= -(F(x) - F^*) + \frac{\lambda}{2}\|x^* - x\|^2_2. \nonumber
    \end{align}
    We now invoke the $\mathcal{K}$-restricted QG condition, $\|x^* - x\|^2_2 \le (2/\gamma_{\mathcal{K}}) (F(x) - F^*)$. Substituting into the quadratic term in the inequality \eqref{ineq:qg} yields
    \begin{align*}
        Q_\lambda(x^*) &\le -(F(x) - F^*) + \frac{\lambda}{2} \left( \frac{2}{\gamma_{\mathcal{K}}} (F(x) - F^*) \right) \\
        &= \left( \frac{\lambda}{\gamma_{\mathcal{K}}} - 1 \right) (F(x) - F^*).
    \end{align*}
    From \eqref{eq: bounding Q}, we substitute this upper bound back into the definition of $\mathcal{D}_g$ and obtain that for any $\lambda>0$,
    \begin{align}
        \mathcal{D}_g(x, \lambda) &= -2\lambda \min_{y\in\mathbb{R}^{d}} Q_\lambda(y) \nonumber \\
        &\ge -2\lambda \left( \frac{\lambda}{\gamma_{\mathcal{K}}} - 1 \right) (F(x) - F^*) \nonumber \\
        &= 2\lambda \left( 1 - \frac{\lambda}{\gamma_{\mathcal{K}}} \right) (F(x) - F^*).\label{eq:max}
    \end{align}
    Notice that \eqref{eq:max} is maximized at $\lambda = \gamma_{\mathcal{K}}/2$. Substituting this optimal step size yields
    \begin{equation*}
        \mathcal{D}_g(x, \gamma_{\mathcal{K}}/2) \ge 2\left(\frac{\gamma_{\mathcal{K}}}{2}\right) \left( 1 - \frac{1}{2} \right) (F(x) - F^*) = \frac{\gamma_{\mathcal{K}}}{2} (F(x) - F^*).
    \end{equation*}
    Since $\mathcal{D}_{g}(x,\alpha)$ is monotonically increasing in $\alpha$ and $L\ge \gamma_{\mathcal{K}}/2$, we obtain
    \[
    \mathcal{D}_g(x, L)
    \ge \frac{\gamma_{\mathcal{K}}}{2} (F(x) - F^*).
    \]
    This is exactly the restricted proximal PL inequality with constant $\nu_{\mathcal{K}} = \gamma_{\mathcal{K}}/2$.
\end{proof}

\section{Proof of Proposition \ref{prop:nonsmooth-polyhedral-pl}}
\label{app:nonsmooth-polyhedral-g:rev}

\begin{proof}

Fix any $x\in \mathcal K\cap\dom(F)$.

\textit{Step 1: Connecting the form convexity with the metric subregularity.} We begin with converting the firm convexity of $g$ relative to $-A^{\top}\xi^{*}$ to the following restricted metric subregularity constant.
\begin{lemma}[Restricted Metric Subregularity]
\label{lem: restricted metric subregularity constant}
Fix $\mathcal{K}$. For a polyhedral regularizer $g$ satisfying the firm convexity with constant $\gamma_{\mathcal{K}}$,
we have
\[
\dist(x,M[\xi^*])
\le
\frac{2}{\gamma_{\mathcal{K}}}\,
\dist\bigl(-A^\top \xi^*,\,\partial g(x)\bigr)
\qquad
\text{for all $x\in\mathcal{K}$}.
\]
\end{lemma}
\begin{proof}[Proof of Lemma \ref{lem: restricted metric subregularity constant}]
Define the tilt perturbation of $g$ by $\tilde{g}(x):= g(x)+\langle A^\top\xi^* , x\rangle.$ Fix \(x\in\mathcal K\) and let \(z\in P_{M[\xi^*]}(x)\), where $P_{M[\xi^*]}(x)$ is the projection of $x$ onto $M[\xi^*]$.
Then we have
$\|x-z\|=\dist(x,M[\xi^*]).$
Take any \(v\in \partial \tilde{g}(x)\). By the subgradient inequality
$\tilde{g}(z)\ge \tilde{g}(x)+\langle v,z-x\rangle$,
we get
\[
\tilde{g}(x)-\tilde{g}(z)\le \langle v,x-z\rangle \le \|v\|\,\|x-z\|.
\]
On the other hand, restricted firm convexity yields $\tilde{g}(x)-\tilde{g}(z)\ge \frac{\gamma_{\mathcal K}}{2}\|x-z\|^2.$ Combining these two inequalities gives $\frac{\gamma_{\mathcal K}}{2}\|x-z\|^2 \le \|v\|\,\|x-z\|.$ 

If \(x\notin M[\xi^*]\), dividing both sides by \(\|x-z\|\) gives
\[
\dist(x,M[\xi^*])=\|x-z\|
\le
\frac{2}{\gamma_{\mathcal K}}\|v\|,
\]
where the inequality is trivial when \(x\in M[\xi^*]\). 
Since $v\in \partial \tilde{g}(x)$ is arbitrary, taking the infimum over $v\in\partial \tilde{g}(x)$ yields $\dist(x,M[\xi^*])
\le
\frac{2}{\gamma_{\mathcal K}} \cdot \dist(0,\partial \tilde{g}(x)).$

Together with \(\partial \tilde{g}(x)=\partial g(x)+A^\top\xi^*\),
this gives
\[
\dist(0,\partial \tilde{g}(x))
=
\dist(-A^\top\xi^*,\partial g(x)),
\]
which proves the claim.
\end{proof}

We now return to the proof of Proposition \ref{prop:nonsmooth-polyhedral-pl}.

\textit{Step 2: Constructing the error bound using the KKT residual.}
Set $r(x):=\dist(0,\partial F(x))$. We have $r(x)=
\dist\bigl(0,\ A^\top \nabla f(Ax)+\partial g(x)\bigr).$ Choose $v_x\in\partial g(x)$ such that
$r(x)=\|A^\top \nabla f(Ax)+v_x\|_2$.
Observe that $\dist(-A^\top \xi^*,\partial g(x))
\le
\|v_x+A^\top \xi^*\|_2.
$ By the triangle inequality $\|v_x+A^\top \xi^*\|_2
\le
\|A^\top \nabla f(Ax)+v_x\|_2
+
\|A^\top(\nabla f(Ax)-\xi^*)\|_2$, we have
$$
\dist(-A^\top \xi^*,\partial g(x))
\le
r(x)+\|A^\top(\nabla f(Ax)-\xi^*)\|_2.
$$
Recall that $\xi^*=\nabla f(y^*)$ and $f(A\cdot)$ is $L$-smooth.
Then we have
$$
\dist(-A^\top \xi^*,\partial g(x))
\le
r(x)+L\,\|Ax-y^*\|_2.
$$
Using Lemma \ref{lem: restricted metric subregularity constant}, we obtain,
\begin{equation}
\label{eq: distance to M upper bound}
\dist(x,M[\xi^*])
\le
\frac{2}{\gamma_{\mathcal K}}\,\dist(-A^\top \xi^*,\partial g(x))
\le
\frac{2}{\gamma_{\mathcal K}}r(x)+\frac{2}{\gamma_{\mathcal K}}L\,\|Ax-y^*\|_2.
\end{equation}
Since $M[\xi^*]=\{z \mid B z\le c\}$ is polyhedral, we have that for any $z\in M[\xi^{*}]$, 
$$
\|(Bx-c)_+\|_2 
\le  \|(Bx-Bz)_{+}\|_{2}
\le \|B\|\,\dist(x,M[\xi^*]).
$$

Thus, combining with \eqref{eq: distance to M upper bound}, we obtain
\begin{equation}
\label{eq: bounding polyhedral violation}
\|(Bx-c)_+\|_2
\le
\|B\|\frac{2}{\gamma_{\mathcal K}}r(x)
+
\|B\|\frac{2}{\gamma_{\mathcal K}}L\,\|Ax-y^*\|_2.
\end{equation}
Now, by assumption,
we have the restricted Hoffman bound on $\mathcal K$:
\begin{equation}
\label{eq: restricted Hoffman bound on the optimal set}
\dist(x,\mathcal X^*)
\le
H_{\mathcal K}\Bigl(\|Ax-y^*\|_2+\|(Bx-c)_+\|_2\Bigr).
\end{equation}
Combining \eqref{eq: bounding polyhedral violation} and \eqref{eq: restricted Hoffman bound on the optimal set} yields
\begin{align}
\label{eq: distance bound 1}
\dist(x,\mathcal X^*)
\le
a_{\mathcal K}\|Ax-y^*\|_2+b_{\mathcal K}r(x),
\end{align}
where
\[
a_{\mathcal K}:=
H_{\mathcal K}\left(1+\|B\|\frac{2}{\gamma_{\mathcal K}}L\right),
\qquad
b_{\mathcal K}:=
H_{\mathcal K}\|B\|\frac{2}{\gamma_{\mathcal K}}.
\]

\textit{Step 3: Constructing a subdifferential error bound.}
Next, let $x^{*}\in P_{\mathcal X^*}(x)$ be a projection of $x$ onto $\mathcal X^*$.
Since $x^{*}\in\mathcal X^*$, we have $Ax^{*}=y^*$ and $-A^\top \xi^*\in \partial g( x^{*})$.

By monotonicity of $\partial g$ (from the convexity of $g$), we have $\langle v_x + A^\top\xi^*,\,x-x^{*}\rangle \ge 0.$ On the other hand, the $\alpha$-strongly convexity of $f$ implies
\[
\langle \nabla f(Ax)-\xi^*,\,Ax-y^*\rangle
=
\langle \nabla f(Ax)-\nabla f(y^*),\,Ax-y^*\rangle
\ge
\alpha \|Ax-y^*\|_2^2.
\]
Adding these two inequalities gives
$
\langle A^\top \nabla f(Ax)+v_x,\ x-x^{*}\rangle
\ge
\alpha \|Ax-y^*\|_2^2.
$
Hence, by the Cauchy--Schwarz inequality, we get:
\[
\alpha \|Ax-y^*\|_2^2
\le
\|A^\top \nabla f(Ax)+v_x\|_2\,\|x-x^{*}\|_2
=
r(x)\,\dist(x,\mathcal X^*).
\]
Together with \eqref{eq: distance bound 1}, we arrive at
$
\alpha \|Ax-y^*\|_2^2
\le
r(x)\bigl(a_{\mathcal K}\|Ax-y^*\|_2+b_{\mathcal K}r(x)\bigr).
$
Now, by setting $t:=\|Ax-y^*\|_2$, we have a quadratic inequality with respect to $t$:
$$
\alpha t^2-a_{\mathcal K}r(x)t-b_{\mathcal K}r(x)^2\le 0.
$$
Solving this quadratic inequality for $t\ge 0$ yields
\[
\|Ax-y^*\|_2 \le \tau_{\mathcal K}\,r(x)
\quad
\text{with}\;
\tau_{\mathcal K}
:=
\frac{a_{\mathcal K}+\sqrt{a_{\mathcal K}^2+4\alpha b_{\mathcal K}}}{2\alpha}.
\]
Substituting back into \eqref{eq: distance bound 1}, we obtain $\dist(x,\mathcal X^*)
\le
\bigl(a_{\mathcal K}\tau_{\mathcal K}+b_{\mathcal K}\bigr)\,r(x).$ Therefore,
letting $\kappa_{\mathcal K}:=a_{\mathcal K}\tau_{\mathcal K}+b_{\mathcal K}$, we get
\begin{equation}
\label{eq: distance bound 2}
\dist(x,\mathcal X^*)
\le
\kappa_{\mathcal K}\,\dist(0,\partial F(x))
\qquad
\text{for all }x\in\mathcal K\cap\dom(F).
\end{equation}

\textit{Step 4: From the subdifferential error bound to the restricted proximal error bound.}
Define the proximal gradient update from $x$ as:
\[
x^+:=\prox_{g/L}\!\left(x-\frac{1}{L} A^{\top} \nabla f(Ax)\right).
\]
Then we have $\mathcal G_{1/L}(x)=L(x-x^+).$ By the first-order optimality condition of the proximal step,
\[
0\in \partial g(x^+) + A^{\top}\nabla f(Ax) + L(x^+-x).
\]
Rearranging yields
\[
\mathcal G_{1/L}(x)- A^{\top}\nabla f(Ax)+ A^{\top}\nabla f(Ax^+)\in \partial F(x^+),
\]
and hence we obtain
\[
\dist(0,\partial F(x^+))
\le
\|\mathcal G_{1/L}(x)- A^{\top}\nabla f(Ax)+A^{\top}\nabla f(Ax^+)\|_2
\le
\|\mathcal G_{1/L}(x)\|_2+\|A^{\top}\nabla f(Ax^+)- A^{\top}\nabla f(Ax)\|_2.
\]
Using the Lipschitz continuity of $\nabla f(A\cdot)$ 
(where $L$ is its Lipschitz constant by assumption)
and the identity
$\|x-x^+\|_2=(1/L)\|\mathcal G_{1/L}(x)\|_2$, we obtain $\dist(0,\partial F(x^+))
\le
2\|\mathcal G_{1/L}(x)\|_2.$

By assumption, the inequality \eqref{eq: distance bound 2} holds for $x^{+}$.
Thus we get
\[
\dist(x^+,\mathcal X^*)
\le
\kappa_{\mathcal K}\,
\dist(0,\partial F(x^{+}))
\le
2\kappa_{\mathcal K}\|\mathcal G_{1/L}(x)\|_2.
\]
Finally, by the triangle inequality, we get
\[
\dist(x,\mathcal X^*)
\le
\|x-x^+\|_2+\dist(x^+,\mathcal X^*)
\le
\frac{1}{L}\|\mathcal G_{1/L}(x)\|_2
+
2\kappa_{\mathcal K}\|\mathcal G_{1/L}(x)\|_2,
\]
which proves the $\mathcal K$-restricted EB inequality (as well as the restricted PL inequality):
\[
\dist(x,\mathcal X^*)
\le
\mu_{\mathcal K}\,\|\mathcal G_{1/L}(x)\|_2
\quad\text{with}\quad
\mu_{\mathcal K}
:=
\frac{1}{L}+2\kappa_{\mathcal K}
\]
\end{proof}

Note that the constant in the restricted EB inequality can be bounded by
$
L^{-1}
+8\alpha^{-1} H_{\mathcal{K}}^{2}\bigl(1+\|B\|\gamma_{\mathcal{K}}L\bigr)^{2}
+8 H_{\mathcal{K}} \|B\| \gamma_{\mathcal{K}}.
$

\section{The Restricted Metric Subregularity Constant for LASSO}
\label{app:Restricted metric subregulairy for LASSO}

In this appendix, 
under a prior parameter restriction and strict complementarity,
we verify the bound
\[
\frac{2}{\gamma_{\mathcal K}}
\le
\max\left\{
\frac{F(\beta^0)}{\eta\,\delta_*},\,
\frac{F(\beta^0)}{2\eta^2}
\right\}
\]
for the LASSO objective
\[
F(\beta)=\frac12\|A\beta-y\|_2^2+\eta\|\beta\|_1,
\qquad \eta>0.
\]

\textit{Step 1: Representation of $M[\hat{\xi}]$.}
Recall the strict complementarity margin from \Cref{def:strict-complementarity}, $\delta_*:=\min_{i \mid \ |\hat{s}_i|<\eta}(\eta-|\hat{s}_i|).$
For each $i\in\{1,\dots,d\}$, let $M_i[\hat{s}_i]
:=
\{t\in\mathbb R\mid \ \hat{s}_i\in \eta\,\partial |t|\}.$
Then the first-order optimality set has the product form:
\[
M[\hat{\xi}]
=
\prod_{i=1}^d M_i[\hat{s}_i]
\]
and the squared distance from $M[\hat{\xi}]$ is separable: $\dist(\beta,M[\hat{\xi}])^{2}=\sum_{i=1}^{d}\dist(\beta_{i},M_{i}[\hat{s}_{i}])^{2}.$

\textit{Step 2: Separable lower bound on the tilt perturbation of $g$.} Define the tilt perturbation of $g$ by
\[
\tilde{g}(\beta):=g(\beta)+\langle A^{\top}\hat{\xi}, \beta\rangle
=\eta \|\beta\|_{1} - \langle \hat{s}, \beta\rangle,
\]
which can be rewritten as
\[
\tilde{g}(\beta)=\eta\sum_{i=1}^{d}|\beta_{i}| -\sum_{i=1}^{d}\hat{s}_{i}\beta_{i}.
\]
Also, by the definition of $M[\hat{\xi}]$, we have $\inf_{\beta\in M[\hat{\xi}]}\tilde{g}(\beta) = 0.$ We thus obtain the representation
\[
\tilde{g}(\beta)-\inf_{\beta\in M[\hat{\xi}]}\tilde{g}(\beta)=\sum_{i=1}^{d}\{
\eta |\beta_{i}|-\hat{s}_{i}\beta_{i}
\}
=\underbrace{\sum_{i\in I_{0}}\{
\eta |\beta_{i}|-\hat{s}_{i}\beta_{i}
\}}_{=:\mathcal{T}_{0}}
+\underbrace{\sum_{i\in I_{+}}\{
\eta |\beta_{i}|-\hat{s}_{i}\beta_{i}
\}}_{=:\mathcal{T}_{+}}
+\underbrace{\sum_{i\in I_{-}}\{
\eta |\beta_{i}|-\hat{s}_{i}\beta_{i}
\}}_{=:\mathcal{T}_{-}}.
\]
The term $\mathcal{T}_{0}$ is bounded using the strict complementarity margin:
\[
\mathcal{T}_{0}\ge \delta_{*}\sum_{i\in I_{0}}|\beta_{i}| = \delta_{*}\sum_{i\in I_{0}}\dist(\beta_{i},M_{i}[\hat{s}_{i}]).
\]
We arrange the terms $\mathcal{T}_{+}$ and $\mathcal{T}_{-}$ by their definition:
\[
\mathcal{T}_{+}+\mathcal{T}_{-} = 2\eta \sum_{i\in I_{+}\cup I_{-}}|\beta_{i}|
=2\eta \sum_{i\in I_{+}\cup I_{-}}\dist(\beta_{i},M_{i}[\hat{s}_{i}]),
\]
and hence we obtain
\begin{equation}
\label{eq: separable lower bound}
\tilde{g}(\beta)-\inf_{\beta\in M[\hat{\xi}]}\tilde{g}(\beta)
\ge 
\delta_{*}\sum_{i\in I_{0}}\dist(\beta_{i},M_{i}[\hat{s}_{i}])
+
2\eta \sum_{i\in I_{+}\cup I_{-}}\dist(\beta_{i},M_{i}[\hat{s}_{i}]).
\end{equation}

\textit{Step 3: From the linear growth in the restricted domain to the quadratic growth.} Recall we have a prior parameter restriction, $\Theta:=\{\beta\in\mathbb{R}^{d}\mid \|\beta\|_{\infty}\le F(\beta_{0})/\eta\}.$ For $\beta\in \Theta$, we have
\[
|\beta|_{i} \ge \frac{\eta}{F(\beta_{0})} |\beta_{i}|^{2},\quad i=1,\ldots,d,
\]
and thus we get
\begin{align}
\begin{split}
\delta_{*}\sum_{i\in I_{0}}\dist(\beta_{i},M_{i}[\hat{s}_{i}])
&\ge \frac{\delta_{*}\eta}{F(\beta_{0})}\sum_{i\in I_{0}}\dist(\beta_{i},M_{i}[\hat{s}_{i}])^{2},\\
2\eta \sum_{i\in I_{+}\cup I_{-}}\dist(\beta_{i},M_{i}[\hat{s}_{i}])
&\ge \frac{2\eta^{2}}{F(\beta_{0})}
\sum_{i\in I_{+}\cup I_{-}}\dist(\beta_{i},M_{i}[\hat{s}_{i}])^{2}.
\label{eq: from linear to quadratic}
\end{split}
\end{align}

\textit{Final step.} Finally, combining \eqref{eq: separable lower bound} with \eqref{eq: from linear to quadratic} yields:
\[
\tilde{g}(\beta)-\inf_{\beta\in M[\hat{\xi}]} \tilde{g}(\beta) \ge \min\left\{
\frac{\eta\delta_{*}}{F(\beta_{0})},
\frac{2\eta^{2}}{F(\beta_{0})}
\right\}
\dist(\beta,M[\hat{\xi}])^{2},
\]
which proves the claim.

\end{document}